\documentclass[a4paper]{amsart}

\usepackage[utf8]{inputenc}
\usepackage[T1]{fontenc}
\usepackage[english]{babel}

\usepackage{amsfonts}
\usepackage{amsmath}
\usepackage[alphabetic]{amsrefs}
\usepackage{amssymb}
\usepackage{amstext}
\usepackage{amsthm}

\usepackage{algorithm}
\usepackage{algpseudocode}

\algrenewcommand\algorithmicrequire{\textbf{Input:}}
\algrenewcommand\algorithmicensure{\textbf{Output:}}

\usepackage{geometry}
\usepackage{fullpage}

\usepackage[shortlabels]{enumitem}
\usepackage{graphicx}
\usepackage{hyperref}
\usepackage{listings}
\usepackage{mathrsfs}
\usepackage{mathtools}
\usepackage{multicol}
\usepackage{stmaryrd}

\usepackage{tikz}
\usepackage{tikz-cd}
\usetikzlibrary{arrows,backgrounds,chains,decorations.pathreplacing,patterns,shapes, calc, positioning}

\newcommand{\inv}{\mathsf{inv}}

\newcommand{\maj}{\mathsf{maj}}

\newtheorem{theorem}{Theorem}[section]
\newtheorem{lemma}[theorem]{Lemma}
\newtheorem{proposition}[theorem]{Proposition}
\newtheorem{corollary}[theorem]{Corollary}
\newtheorem{conjecture}[theorem]{Conjecture}

\newtheorem{manualtheoreminner}{Theorem}
\newenvironment{manualtheorem}[1]{%
	\renewcommand\themanualtheoreminner{#1}%
	\manualtheoreminner
}{\endmanualtheoreminner}

\newtheorem{manualconjinner}{Conjecture}
\newenvironment{manualconj}[1]{%
	\renewcommand\themanualconjinner{#1}%
	\manualconjinner
}{\endmanualconjinner}

\theoremstyle{definition}
\newtheorem{definition}[theorem]{Definition}

\newtheorem{example}[theorem]{Example}

\theoremstyle{remark}
\newtheorem{remark}[theorem]{Remark}

\numberwithin{equation}{section}

\title{Chromatic functions, interval orders and increasing forests}

\author{Michele D'Adderio}
\address{Universit\`a di Pisa\\Dipartimento di Matematica\\ Largo Bruno Pontecorvo 5, 56127 Pisa\\ Italy}\email{michele.dadderio@unipi.it}

\author{Roberto Riccardi}
\address{Scuola Normale Superiore\\
	Piazza dei Cavalieri 7,
	56126 Pisa\\ Italy}\email{roberto.riccardi@sns.it}

\author{Viola Siconolfi}
\address{Politecnico di Bari\\Dipartimento di Meccanica, Matematica e Managment\\ Via Orabona 4, 70125 Bari\\ Italy}\email{viola.siconolfi@poliba.it}

\begin{document}
	
\begin{abstract}
The chromatic quasisymmetric functions (csf) of Shareshian and Wachs associated to unit interval orders have attracted a lot of interest since their introduction in 2016, both in combinatorics and geometry, because of their relation to the famous Stanley-Stembridge conjecture (1993) and to the topology of Hessenberg varieties, respectively. 

In the present work we study the csf associated to the larger class of interval orders with no restriction on the length of the intervals. Inspired by an article of Abreu and Nigro, we show that these csf are weighted sums of certain quasisymmetric functions associated to the increasing spanning forests of the associated incomparability graphs. Furthermore, we define quasisymmetric functions that include the unicellular LLT symmetric functions and generalize an identity due to Carlsson and Mellit. Finally we conjecture a formula giving their expansion in the type 1 power sum quasisymmetric functions which should extend a theorem of Athanasiadis.  
\end{abstract}

\maketitle
\tableofcontents

\section{Introduction}

In \cite{Shareshian_Wachs_Advances} Shareshian and Wachs introduced the \emph{chromatic quasisymmetric function} $\chi_G[X;q]$ associated to every graph $G$ whose vertices are totally ordered. At $q=1$ the series $\chi_G[X;q]$ reduces to the well-known chromatic symmetric function $\chi_G[X;1]=\chi_G(x)$ introduced by Stanley in \cite{Stanley_Chrom_Sym}. A famous conjecture of Stanley and Stembridge (\cite{Stanley_Chrom_Sym}*{Conjecture~5.1}, \cite{Stanley_Stembridge}*{Conjecture~5.5}) states that if $G$ is the incomparability graph of a $(\mathbf{3}+\mathbf{1})$-free poset, then $\chi_G[X;1]$ is $e$-positive, i.e.\ its expansion in the elementary symmetric functions has coefficients in $\mathbb{N}$. Shareshian and Wachs showed (cf.\ \cite{Shareshian_Wachs_Advances}*{Theorem~4.5}) that if $G$ is the incomparability graph of a poset that is both $(\mathbf{3}+\mathbf{1})$-free and $(\mathbf{2}+\mathbf{2})$-free, then $\chi_G[X;q]$ is a symmetric function, and they conjecture that it is $e$-positive, i.e.\ its expansion in the elementary symmetric functions has coefficients in $\mathbb{N}[q]$. Thanks to a result of Guay-Paquet \cite{guaypaquet_modular_law}, it is known that the Shareshian-Wachs conjecture implies the Stanley-Stembridge conjecture.

The posets that are  $(\mathbf{3}+\mathbf{1})$-free and $(\mathbf{2}+\mathbf{2})$-free are precisely the \emph{unit interval orders} (see \cite{Scott_Suppes}), whose elements are intervals in $\mathbb{R}$ of the same length, and an interval $a$ is smaller than an interval $b$ if all the points of $a$ are strictly smaller than all the points of $b$. If in such a poset we order the intervals increasingly according to their left endpoints, then we get a total order on them, and now the incomparability graphs of these posets will inherit this total order on the vertices, giving the labelled graphs $G$ involved in the Shareshian-Wachs conjecture. In the present article we will call these labelled graphs \emph{Dyck graphs}, as they are in a natural bijection with Dyck paths: see Figure~\ref{fig:Dyck_graph} for an example.

Since their introduction, the symmetric functions $\chi_G[X;q]$ of Dyck graphs have been extensively studied, not only for their connection to the Stanley-Stembridge conjecture (cf.\ \cite{Huh_Nam_Yoo,Abreu_Nigro_Modular_Law,Skandera,Cho_Hong,Nadeau_Tewari_Downup,Colmenarejo_Morales_Panova}, just to mention a few recent articles), but also for their connection to the topology of Hessenberg varieties: for the latter see for example \cite{Shareshian_Wachs_Advances,guaypaquet_second_pf,Brosnan_Chow,Abe_Horiguchi_Survey}. We will not say any more things about the relation to geometry, and we will limit ourselves to recall here a few results, belonging more properly to algebraic combinatorics, that will be relevant to the present article.

\smallskip

In \cite{Carlsson-Mellit-ShuffleConj-2015}, Carlsson and Mellit proved a surprising formula relating the symmetric functions $\chi_G[X;q]$ of Dyck graphs to the so called \emph{unicellular LLT symmetric functions} $\mathrm{LLT}_G[X;q]$ (see e.g.\ \cite{Alexandersson-Panova}*{Section~3} for their connection with the original symmetric functions introduced by Lascoux, Lecrerc and Thibon), which can be restated as follows using \emph{plethystic notation} (for which we refer to \cite{Loehr-Remmel-plethystic-2011}).

\begin{theorem}[Carlsson, Mellit]
If $G$ is a Dyck graph on $n$ vertices, then
\begin{equation} \label{eq:CM_identity}
(1-q)^n  \omega \chi_G\left[X\frac{1}{1-q}\right] =\mathrm{LLT}_G[X;q].
\end{equation}
\end{theorem}

In \cite{Abreu_Nigro_Forests}, Abreu and Nigro prove a deep formula for $\chi_G[X;q]$ of Dyck graph $G$ as a weighted sum over the \emph{increasing spanning forests}  of $G$. The role of the increasing spanning forests in the combinatorics of these graphs was already highlighted in \cite{Hallam_Sagan}, where in fact the authors work with more general graphs. Abreu and Nigro combined their formula with the identity \eqref{eq:CM_identity} of Carlsson and Mellit to give a new proof of a formula for the expansion of $\mathrm{LLT}_G[X;q+1]$ (here $G$ is a Dyck graph) in the elementary symmetric function basis that was conjectured by Alexandersson \cite{Alexandersson_LLT_conj} and independently by Garsia, Haglund, Qiu and Romero \cite{Garsia_Haglund_Qiu_Romero}, and first proved by Alexandersson and Sulzgruber \cite{Alexandersson_Sulzgruber} (the $e$-positivity was already proved in \cite{DAdderio_e_positivity}).

Finally, in \cite{Athanasiadis} Athanasiadis proved a combinatorial formula (cf.\ Theorem~\ref{thm:athanasiadis}) giving the expansion of $\chi_G[X;q]$ (when $G$ is a Dyck graph) in the power symmetric functions, that was conjectured in \cite{Shareshian_Wachs_Advances}*{Conjecture~7.6}. The starting point of his proof is the formula \cite{Shareshian_Wachs_Advances}*{Theorem~6.3} giving the expansion of $\chi_G[X;q]$ (when $G$ is a Dyck graph) in the Schur function basis.

We now turn to the contributions of the present article (all the missing definitions will be provided later in the text).

\smallskip

If from the definition of unit interval orders we drop the condition on the intervals to have all the same length, then we get the \emph{interval orders}, and these are precisely the $(\mathbf{2}+\mathbf{2})$-free posets (see \cite{Fishburn}). If we order again such intervals increasingly according to their left endpoints, then we get a total order on them, and now the incomparability graphs of these posets will inherit this total order on the vertices: these labelled graphs $G$ will be called (somewhat improperly) \emph{interval graphs} in this article, and their chromatic quasisymmetric functions $\chi_G[X;q]$ are the object of our study: see Figure~\ref{fig:interval_graph} for an example.

First of all we observe that Dyck graphs are interval graphs, but typically if $G$ is an interval graph that is not a Dyck graph, then $\chi_G[X;q]$ is not a symmetric function. Hence we will be naturally working with quasisymmetric functions.

Inspired by the work of Abreu and Nigro \cite{Abreu_Nigro_Forests}, given an interval graph $G$ we will define a surjective function $\Phi_G$ from the (infinite) set $\mathsf{PC}(G)$ of proper colorings  $\kappa:V(G)\to \mathbb{Z}_{>0}$ of $G$ to the (finite) set $\mathsf{ISF}(G)$ of the increasing spanning forests of $G$ (see Algorithm~\ref{alg:main}). This will allow us to define for every forest $F\in \mathsf{ISF}(G)$ the associated series
\[\mathcal{Q}_F^{(G)}:=\mathop{\sum_{\kappa\in \mathsf{PC}(G)}}_{\Phi_G(\kappa)=F}x_\kappa,\]
where $x_\kappa=\prod_{v\in V(G)}x_{\kappa(v)}$ and $x_1,x_2,\dots$ are variables.

We will show that these are quasisymmetric functions, and we will prove the following formula for their expansion in the fundamental (Gessel) quasisymmetric function basis $\{L_{n,S}\mid n\in \mathbb{N},S\subseteq [n-1]\}$.
\begin{manualtheorem}{\ref{thm:forest_fundamental}} 
	Given an interval graph $G$ on $n$ vertices and $F\in \mathsf{ISF}(G)$, we have
\begin{equation*}
\mathcal{Q}_F^{(G)}=\mathop{\sum_{\sigma\in \mathfrak{S}_n}}_{\mathsf{CoInv}_G(\sigma)=\mathsf{CoInv}_G(F)}L_{n,\mathsf{Des}_G(\sigma^{-1})}.
\end{equation*}
\end{manualtheorem}
Using the weight $\mathsf{wt}_G(F)$ introduced in \cite{Abreu_Nigro_Forests} (in fact a slight modification), we will prove the following formula for $\chi_G[X;q]$ of an interval graph $G$:
\begin{manualtheorem}{\ref{thm:chrom_forests_formula}} 
	Given an interval graph $G$, we have
\begin{equation*}
\chi_G[X;q]=\sum_{F\in \mathsf{ISF}(G)}q^{\mathsf{wt}_G(F)}\mathcal{Q}_F^{(G)}.
\end{equation*}
\end{manualtheorem}

For every simple graph $G$ with totally ordered vertices we introduce the quasisymmetric function 
\[\mathrm{LLT}_G[X;q]:=\sum_{\kappa \in \mathsf{C}(G)}q^{\mathsf{inv}_G(\kappa)}x_\kappa,\]
where the sum is over all (not necessarily proper) colorings $\kappa:V(G)\to \mathbb{Z}_{>0}$ of $G$. When $G$ is a Dyck graph, these are precisely the corresponding unicellular LLT symmetric functions. For interval graphs that are not Dyck graphs, these are typically not symmetric functions, but always quasisymmetric. Indeed the following general formula holds.
\begin{manualtheorem}{\ref{thm:LLT_fundamental}}
	Given any simple graph $G$ on $n$ vertices which are totally ordered, we have
	\[\mathrm{LLT}_G[X;q]=\sum_{\sigma\in \mathfrak{S}_n}q^{\mathsf{inv}_G(\sigma)}L_{n,\mathsf{Des}(\sigma^{-1})}.\]
\end{manualtheorem}
In the previous theorem the ``$L$'' are the fundamental (Gessel) quasisymmetric functions.

Recall the well-known involutions $\rho$ and $\psi$ of the algebra $\mathrm{QSym}$ of quasisymmetric functions defined by $\psi(L_{\alpha}):=L_{\alpha^c}$ and $\rho(L_{\alpha})=L_{\alpha^r}$. 

The main result of this article is the following theorem.

\begin{manualtheorem}{\ref{thm:main_theorem}}
	Given $G$ an interval graph on $n$ vertices, we have
	\begin{equation*} 
	(1-q)^n \rho \left(\psi \chi_G\left[X\frac{1}{1-q}\right]\right)=\mathrm{LLT}_G[X;q].
	\end{equation*}
\end{manualtheorem}

Notice that the plethysm of quasisymmetric functions needs a careful definition, since it depends on the order on the alphabet (cf.\ \cite{Loehr-Remmel-plethystic-2011}).

Since $\psi$ restricts to the involution $\omega$ on the algebra $\mathrm{Sym}\subset \mathrm{QSym}$ of symmetric functions, $\rho$ restricts to the identity operator on $\mathrm{Sym}$, and the plethysm of quasisymmetric functions restricts to the plethysm of symmetric functions, this is indeed a generalization of the identity \eqref{eq:CM_identity} of Carlsson and Mellit. We remark that our argument gives an independent proof of that same identity. 

Our proof of Theorem~\ref{thm:main_theorem} is based on a formula which in turn relies on a result of Kasraoui \cite{Kasraoui_maj-inv}.

In \cite{Ballantine_et_al} the authors study a family of quasisymmetric functions that they call \emph{type 1 quasisymmetric power sums}, and they denote $\Psi_\alpha$. Actually $\{\Psi_\alpha\mid \alpha\text{ composition}\}$ is a basis of $\mathrm{QSym}$ that refines the power symmetric function basis.

We state the following conjecture for the expansion of our $\mathcal{Q}_{F}^{(G)}$ in the $\Psi_\alpha$.

\begin{manualconj}{\ref{conj:forest_psi_expansion}}
	For any interval graph $G$ on $n$ vertices and any permutation $\tau\in \mathfrak{S}_n$ (thought of as a proper coloring of $G$) we have
	\begin{equation*} 
	\rho \psi \mathcal{Q}_{\Phi_G(\tau)}^{(G)}=\sum_{\alpha\vDash n}\frac{\Psi_\alpha}{z_{\alpha}}\# \{ \sigma \in \mathcal{N}_{G,\alpha}\mid \mathsf{CoInv}_G(\sigma^{-1})=\mathsf{Inv}_G(\tau)\}.
	\end{equation*}
\end{manualconj}
We will show that this conjecture implies the following one, which is supposed to provide an extension of the formula proved by Athanasiadis in \cite{Athanasiadis}.

\begin{manualconj}{\ref{conj:interval_psi_exp}}
	For any interval graph $G$ on $n$ vertices we have
	\[\rho\psi \chi_G[X;q]=\sum_{\alpha\vDash n}\frac{\Psi_\alpha}{z_{\alpha}}\sum_{\sigma \in \mathcal{N}_{G,\alpha}}q^{\widetilde{\mathsf{inv}}_G(\sigma)}.\]
\end{manualconj}

\medskip

The rest of the present article is organized in the following way. In Section~2 we introduce both the interval graphs and the Dyck graphs, and clarify their relation with interval orders. In Section~3 we recall some background from quasisymmetric and symmetric functions. In particular we will introduce our plethysm of quasisymmetric functions and we will provide a couple of formulas that we will need later. In Section~4 we prove two basic formulas about colorings and their inversions, while in Section~5 we analyze the case of interval graphs, proving related formulas for proper colorings and (co)inversions. In Section~6 we introduce the increasing spanning forests of interval graphs $G$, their weight $\mathsf{wt}_G$, the function $\Phi_G$, and we study the first properties of the quasisymmetric functions $\mathcal{Q}_F^{(G)}$, including Theorem~\ref{thm:forest_fundamental}. In Section~7 we introduce chromatic quasisymmetric functions and LLT quasisymmetric functions associated to interval graphs, and we apply results from previous sections to prove Theorem~\ref{thm:chrom_forests_formula} and Theorem~\ref{thm:LLT_fundamental}. In Section~8 we prove a fundamental formula that is at the heart of our proof of Theorem~\ref{thm:main_theorem}, which is based on a result of Kasraoui \cite{Kasraoui_maj-inv} that we will explain in Section~9 for completeness.  In Section~10 we will use the formula from Section~8 and several results from previous sections to deduce Theorem~\ref{thm:main_theorem}. In Section~11 we will state Conjecture~\ref{conj:forest_psi_expansion} and Conjecture~\ref{conj:interval_psi_exp}, and show how the former implies the latter. Finally in Section~12 we add some speculative comments.

\medskip

\noindent \emph{Acknowledgements}. This article stemmed from the suspicion that the identity \eqref{eq:CM_identity} of Carlsson and Mellit could hold in greater generality. Once we noticed that for the interval graphs there was a generalization, we realized that these graphs were mentioned in the introduction of the article of Abreu and Nigro. We thank these four authors for their insightful mathematical works. We are happy to thank also Philippe Nadeau for sharing with us his unpublished work with Vasu Tewari suggesting the direct arguments in Sections~4 and 5.

\section{Interval graphs}

In this article a graph will always be simple, i.e.\ no loops and no multiple edges.

For every positive integer $n\in \mathbb{Z}_{>0}:=\{1,2,3,\dots\}$ we will use the notation $[n]:=\{1,2,\dots,n\}$. 

Consider a \emph{graph} $G=([n],E)$, where $E=E(G)$ is the set of \emph{edges} of $G$, while $[n]=V(G)$ is the set of \emph{vertices} of $G$. 
\begin{remark}
We can think of our $G=([n],E)$ as a graph on $n$ vertices that are \emph{labelled} by the elements of $[n]$, each appearing exactly once. In fact in this work we will simply use the fact that the vertices of $G$ have a fixed total order. We prefer to work directly with $V(G)=[n]$, in order to keep the notation lighter. In particular we will systematically omit the word \emph{labelled}, and we will talk simply about \emph{graphs}. A more precise name would probably be ``naturally labelled graphs'', but its systematic use would make the reading much more unpleasant.
\end{remark}
We will think of an edge of $G$ both as a $2$-subset $\{i,j\}\in E$ and as an ordered pair $(i,j)\in E$ with $i<j$.

In this work a (\emph{labelled}) \emph{graph} $G=([n],E)$ will be called \emph{interval} if whenever $\{i,j\}\in E$ and $i<j$, then $\{i,k\}\in E$ for every $i<k\leq j$. We will call $\mathcal{IG}_n$ the set of all interval graphs with vertex set $[n]$.
\begin{remark}
Again, recall that our graphs are actually \emph{labelled}, so for example the graphs $([3],\{(1,2)\})$ and $([3],\{(2,3)\})$ are two distinct elements of $\mathcal{IG}_3$, though they are clearly isomorphic as abstract graphs.

The name interval graphs will be justified shortly in this section.
\end{remark}

We can represent an interval graph $G=([n],E)$ in the following way: in a $n\times n$ square grid we order the columns from left to right with numbers $1,2,\dots,n$ and similarly the rows from bottom to top; then we color the cells $(i,j)\in E$ (hence $i<j$). See Figure~\ref{fig:interval_graph} for an example.

\begin{figure*}[!ht]
	\begin{tikzpicture}[scale=0.5]
		\draw[gray!60, thin](0,0) grid (8,8);
		\filldraw[red, opacity=0.3] (1,2) -- (2,2)--(2,7)--(1,7)-- cycle;
		\filldraw[red, opacity=0.3] (0,1) -- (0,3)--(1,3)--(1,1)-- cycle;
		\filldraw[red, opacity=0.3] (2,3) -- (2,6)--(3,6)--(3,3)-- cycle;
		\filldraw[red, opacity=0.3] (4,5) -- (4,7)--(5,7)--(5,5)-- cycle;
		\filldraw[red, opacity=0.3] (5,6) -- (5,8)--(7,8)--(7,7)--(6,7)--(6,6)-- cycle;
		\node at (0.5,0.5) {$1$};	
		\node at (1.5,1.5) {$2$};
		\node at (2.5,2.5) {$3$};
		\node at (3.5,3.5) {$4$};
		\node at (4.5,4.5) {$5$};
		\node at (5.5,5.5) {$6$};
		\node at (6.5,6.5) {$7$};
		\node at (7.5,7.5) {$8$};
		
	\end{tikzpicture}
	\caption{The interval graph $G=([8],\{(1,2),(1,3),(2,3),(2,4),(2,5),(2,6),(2,7),$ $(3,4),(3,5),(3,6),(5,6),(5,7),(6,7),(6,8),(7,8)\})$.} \label{fig:interval_graph}
\end{figure*}
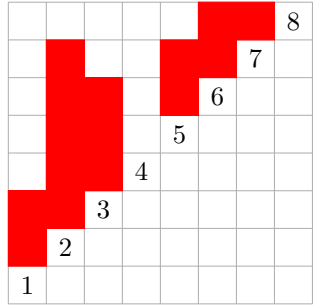

Notice that in this pictures we simply obtain a bunch of (possibly empty) colored columns, starting just above the diagonal cells. Hence clearly $|\mathcal{IG}_n|=n!$.

Given an interval graph $G\in \mathcal{IG}_n$, we can consider its \emph{flipped}, obtained from $G$ by replacing each edge $\{i,j\}$ with an edge $\{n+1-i,n+1-j\}$: in terms of pictures, this corresponds to flip the picture of $G$ around the line $y=-x$.

An interval graph  $G\in \mathcal{IG}_n$ such that its flipped is still in $\mathcal{IG}_n$ is called a \emph{Dyck graph}. The explanation of the name is obvious, since the picture of a Dyck graph determines a Dyck path: see Figure~\ref{fig:Dyck_graph} for an example. Hence, if we denote the set of Dyck graphs with $n$ vertices by $\mathcal{DG}_n$, then clearly $|\mathcal{DG}_n|$ is the $n$-the Catalan number $\binom{2n}{n}/(n+1)$. 

\begin{figure*}[!ht]
	\begin{tikzpicture}[scale=0.5]
		\draw[gray!60, thin](0,0) grid (8,8);
		\filldraw[red, opacity=0.3] (1,2) -- (2,2)--(2,4)--(1,4)-- cycle;
		\filldraw[red, opacity=0.3] (0,1) -- (0,3)--(1,3)--(1,1)-- cycle;
		\filldraw[red, opacity=0.3] (2,3) -- (2,4)--(3,4)--(3,3)-- cycle;
		\filldraw[red, opacity=0.3] (4,5) -- (4,7)--(5,7)--(5,5)-- cycle;
		\filldraw[red, opacity=0.3] (5,6) -- (5,7)--(6,7)--(6,8)--(7,8)--(7,7)--(6,7)--(6,6)-- cycle;
		\draw[blue!60, line width=1.6pt] (0,0)--(0,3)--(1,3)--(1,4)--(4,4)--(4,7)--(6,7)--(6,8)--(8,8);
		\node at (0.5,0.5) {$1$};	
		\node at (1.5,1.5) {$2$};
		\node at (2.5,2.5) {$3$};
		\node at (3.5,3.5) {$4$};
		\node at (4.5,4.5) {$5$};
		\node at (5.5,5.5) {$6$};
		\node at (6.5,6.5) {$7$};
		\node at (7.5,7.5) {$8$};
		
	\end{tikzpicture}
	\caption{The Dyck graph $G=([8],\{(1,2),(1,3),(2,3),(2,4),(3,4),(5,6),(5,7),(6,7),(7,8)\})$.} \label{fig:Dyck_graph}
\end{figure*}
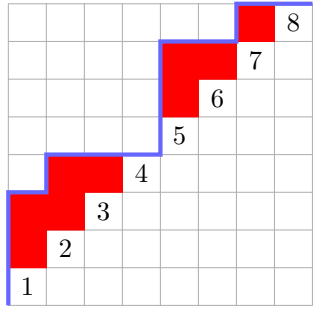

\begin{remark} \label{rem:interval_clique}
If $(G=([n],E))\in \mathcal{IG}_n$ is an interval graph, then $G$ has the following property: for every $j\in [n]$, the set of $i\in [j-1]$ such that $(i,j)\in E$ form a \emph{clique} in $G$, namely all these vertices are connected with each other in $G$. Indeed, if $i_1<i_2<\cdots <i_r<j$ are the aforementioned vertices, then for every $k\in [r-1]$, $(i_k,j)\in E$ implies $(i_k,i_t)$ for every $k<t\leq r$.

In fact, this property is saying that in our interval graphs the natural order on the set of vertices $[n]$ is a \emph{perfect elimination order}, showing that our interval graphs are \emph{chordal} (see \cite{Fulkerson_Gross}).
\end{remark}

\smallskip

It turns out that the interval graphs are the incomparability graphs of certain posets called \emph{interval orders} (hence their name).

Given a (naturally labelled) poset $P=([n],<_P)$, its \emph{incomparability (labelled) graph} $\mathrm{Inc(P)}=([n],E_P)$ is defined by setting $\{i,j\}\in E_P$ if and only if $i$ and $j$ are incomparable in $P$.

Let $\mathcal{I}$ be the set of all bounded closed intervals of $\mathbb{R}$, and given $I=[a,b]$ and $J=[c,d]$ we set $I \prec J$ if and only if $b < c$. Clearly $(\mathcal{I}, \prec)$ is a poset. Any subposet of $(\mathcal{I}, \prec)$ is called an \emph{interval order}. An example of interval order is given in Figure~\ref{fig:poset_intervals}: on the right there are the intervals, while on the left there is the corresponding Hasse diagram (in which node $i$ corresponds to interval $I_i$ for every $i$).

The following propositions are probably well known. We sketch their proofs for completeness.

\begin{proposition} \label{prop:interval_poset_graphs}
	For every $\mathcal{H} \subseteq \mathcal{I}$ such that $|\mathcal{H}|=n \in \mathbb{N}$, there exists an order-preserving bijection $\phi:(\mathcal{H}, \prec) \longrightarrow ([n],<)$ such that $\mathrm{Inc}(\phi(\mathcal{H}, \prec)) \in \mathcal{IG}_n$.
\end{proposition}

\begin{proof}[Sketch of proof]
We simply order increasingly $H_1,H_2,\dots$ the elements of $\mathcal{H}$ according to their leftmost points: it is easy to see that if $i<j$ and $H_i$ and $H_j$ are incomparable, i.e.\ they overlap, then every $H_k$ with $i<k\leq j$ must overlap with $H_i$ as well. For example, the intervals on the right of Figure~\ref{fig:poset_intervals} are ordered as we just explained, and they correspond to the interval graph in Figure~\ref{fig:interval_graph}.
\end{proof}

On the contrary, every graph in $\mathcal{IG}_n$ is the incomparability graph of a finite subposet of $\mathcal{I}$.

\begin{proposition} \label{prop:Interval_graph_to_intervals}
	For every $G=([n], E) \in \mathcal{IG}_n$ there exist a finite subset $\mathcal{H} \subseteq \mathcal{I}$ and an order-preserving bijection $\phi: (\mathcal{H},\prec) \longrightarrow ([n],<)$ such that $\mathrm{Inc}(\phi(\mathcal{H}, \prec))=G$.
\end{proposition}

\begin{proof}[Sketch of proof]
	Given $G=([n], E) \in \mathcal{IG}_n$, let $P=([n], <_P)$ be the poset on $[n]$ such that $\mathrm{Inc}(P)=G$: more precisely $i<_P j$ if and only if $i<j$ and $\{i,j\}\notin E$. It is easy to check that this gives indeed a poset: for example, the poset on the left of Figure~\ref{fig:poset_intervals} corresponds to the interval graph in Figure~\ref{fig:interval_graph}. 
	
	For each $j \in [n]$, consider 
	\[M_j:=\lbrace k \in [n] \hspace{1mm} | \hspace{1mm} j <_P k \rbrace. \]
	Now we take $\mathcal{H}= \lbrace I_1, I_2, \ldots, I_n \rbrace \subseteq \mathcal{I}$, where for each $j \in [n]$ we define $I_j$ in the following way:
	\[
	I_j=
	\begin{cases}
		[j,n+1], & \text{ if } M_j= \emptyset \\
		[j, k_j - \frac{1}{2}], & \text{ otherwise, where } k_j=\min M_j
	\end{cases}.
	\]
	It is easy to see that $\phi: \mathcal{H} \longrightarrow [n]$ such that $\phi(I_j)=j$ for each $j \in [n]$ gives the desired bijection. For example, the intervals on the right of Figure~\ref{fig:poset_intervals} correspond to the interval graph in Figure~\ref{fig:interval_graph} via $\phi$.
\end{proof}

\begin{figure}[ht]
	\centering
	\begin{tikzpicture}[auto,node distance=3cm,
		thick,root node/.style={shape=circle, inner sep=3pt, fill=red,draw},main node/.style={shape=circle, inner sep=3pt, fill=black,draw},external node/.style={shape=circle, inner sep=6pt,draw},scale=.35]
		
		\node[external node] (1) at (-3,0) {};
		\node[external node] (2) at (3,0) {};
		\node[external node] (3) at (0,0) {};
		\node[external node] (4) at (-3,3) {};
		\node[external node] (5) at (-6,6) {};
		\node[external node] (6) at (-3,6) {};
		\node[external node] (7) at (0,6) {};
		\node[external node] (8) at (0,11) {};
		
		\node at (0,0) {$3$};
		\node at (-3,0) {$1$};
		\node at (3,0) {$2$};
		\node at (-3,3) {$4$};
		\node at (-6,6) {$5$};
		\node at (-3,6) {$6$};
		\node at (0,6) {$7$};
		\node at (0,11) {$8$};
		
		\draw[gray!90, line width=0.32mm] (6,0)--(22,0);
		\draw[gray!90, line width=0.32mm] (6,-0.3)--(6,0.3);
		\draw[gray!90, line width=0.32mm] (8,-0.3)--(8,0.3);
		\draw[gray!90, line width=0.32mm] (10,-0.3)--(10,0.3);
		\draw[gray!90, line width=0.32mm] (12,-0.3)--(12,0.3);
		\draw[gray!90, line width=0.32mm] (14,-0.3)--(14,0.3);
		\draw[gray!90, line width=0.32mm] (16,-0.3)--(16,0.3);
		\draw[gray!90, line width=0.32mm] (18,-0.3)--(18,0.3);
		\draw[gray!90, line width=0.32mm] (20,-0.3)--(20,0.3);
		\draw[gray!90, line width=0.32mm] (22,-0.3)--(22,0.3);
		
		\draw[black, line width=0.4mm] (20,1)--(22,1);
		\draw[black, line width=0.4mm] (20,1-0.3)--(20,1+0.3);
		\draw[black, line width=0.4mm] (22,1-0.3)--(22,1+0.3);
		\node at (21,1.5) {$I_8$};
		
		\draw[black, line width=0.4mm] (18,1)--(19,1);
		\draw[black, line width=0.4mm] (18,1-0.3)--(18,1+0.3);
		\draw[black, line width=0.4mm] (19,1-0.3)--(19,1+0.3);
		\node at (18.5,1.5) {$I_7$};
		
		\draw[black, line width=0.4mm] (16,2.5)--(22,2.5);
		\draw[black, line width=0.4mm] (16,2.5-0.3)--(16,2.5+0.3);
		\draw[black, line width=0.4mm] (22,2.5-0.3)--(22,2.5+0.3);
		\node at (19,3) {$I_6$};
		
		\draw[black, line width=0.4mm] (14,4)--(19,4);
		\draw[black, line width=0.4mm] (14,4-0.3)--(14,4+0.3);
		\draw[black, line width=0.4mm] (19,4-0.3)--(19,4+0.3);
		\node at (16.5,4.5) {$I_5$};
		
		\draw[black, line width=0.4mm] (12,4)--(13,4);
		\draw[black, line width=0.4mm] (12,4-0.3)--(12,4+0.3);
		\draw[black, line width=0.4mm] (13,4-0.3)--(13,4+0.3);
		\node at (12.5,4.5) {$I_4$};
		
		\draw[black, line width=0.4mm] (10,5.5)--(17,5.5);
		\draw[black, line width=0.4mm] (10,5.5-0.3)--(10,5.5+0.3);
		\draw[black, line width=0.4mm] (17,5.5-0.3)--(17,5.5+0.3);
		\node at (13.5,6) {$I_3$};
		
		\draw[black, line width=0.4mm] (8,7)--(19,7);
		\draw[black, line width=0.4mm] (8,7-0.3)--(8,7+0.3);
		\draw[black, line width=0.4mm] (19,7-0.3)--(19,7+0.3);
		\node at (13.5,7.5) {$I_2$};
		
		\draw[black, line width=0.4mm] (6,8.5)--(11,8.5);
		\draw[black, line width=0.4mm] (6,8.5-0.3)--(6,8.5+0.3);
		\draw[black, line width=0.4mm] (11,8.5-0.3)--(11,8.5+0.3);
		\node at (8.5,9) {$I_1$};
		
		\node[gray!90] at (6,0-0.8) {$1$};
		\node[gray!90] at (8,0-0.8) {$2$};
		\node[gray!90] at (10,0-0.8) {$3$};
		\node[gray!90] at (12,0-0.8) {$4$};
		\node[gray!90] at (14,0-0.8) {$5$};
		\node[gray!90] at (16,0-0.8) {$6$};
		\node[gray!90] at (18,0-0.8) {$7$};
		\node[gray!90] at (20,0-0.8) {$8$};
		\node[gray!90] at (22,0-0.8) {$9$};
		
		\path
		(1) edge node {} (4)
		(4) edge node {} (5)
		(4) edge node {} (6)
		(4) edge node {} (7)
		(5) edge node {} (8)
		(7) edge node {} (8)
		(3) edge node {} (7)
		(2) edge node {} (8);
	\end{tikzpicture}
	\caption{On the left the poset whose incomparability graph is the $G$ in Figure~\ref{fig:interval_graph}. On the right the family of intervals associated to the $G$ in Figure~\ref{fig:interval_graph} by the proof of Proposition~\ref{prop:Interval_graph_to_intervals}.}
	\label{fig:poset_intervals}
\end{figure}
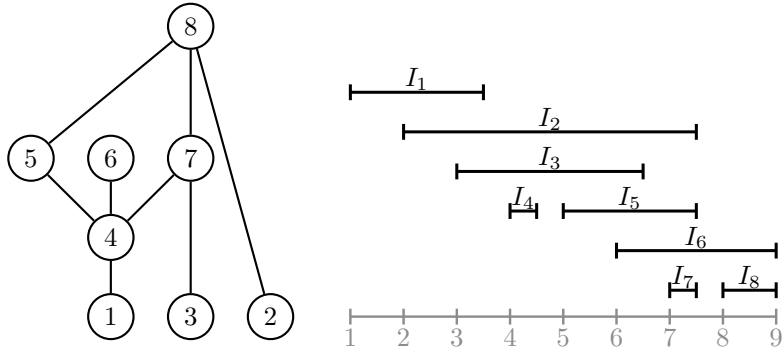

It is known (see \cite{Shareshian_Wachs_Advances}*{Section~4}) that the Dyck graphs are the incomparability graphs of (finite) unit interval orders, i.e.\ interval orders where every interval has the same fixed length (without loss of generality equal to $1$). 
\begin{remark}
In one direction, notice that in the proof of Proposition~\ref{prop:interval_poset_graphs} that we sketched, we could have ordered the intervals from right to left, i.e.\ according to their right endpoint: of course in general we would find again an interval graph, but in the special case where all the intervals have the same length, this order would correspond precisely to the order of the flipped graph, showing that also its flipped is an interval graph, hence showing that it is actually a Dyck graph. 
\end{remark}

To conclude this section, we recall that the interval orders are precisely the $(\mathbf{2}+\mathbf{2})$-free posets (see \cite{Fishburn}), and that the unit interval orders are precisely the $(\mathbf{3}+\mathbf{1})$-free and $(\mathbf{2}+\mathbf{2})$-free posets (see \cite{Scott_Suppes}). Finally, we observe here that our interval graphs are precisely the ``dually factorial posets'' appearing in \cite{Claesson_Linusson}.

\section{Symmetric and quasisymmetric functions} \label{sec:qsym}

In this section we recall a few basic facts of symmetric and quasisymmetric functions, mainly to fix the notation. As general references we suggest \cite{Luoto_Mykytiuk_Willigenburg_Book} for quasisymmetric functions, \cite{Stanley-Book-1999}*{Chapter~7} for both symmetric and quasisymmetric functions, and \cite{Stanley_Book_2012} for posets. 

Given a composition $\alpha=(\alpha_1,\alpha_2,\dots,\alpha_k)$ of $n\in \mathbb{N}$ (denoted $\alpha\vDash n$), we denote its \emph{size} by $|\alpha|=\sum_i\alpha_i=n$ and its \emph{length} by $\ell(\alpha)=k$. For brevity, sometimes we will use the \emph{exponential notation}, so that for example we will write $(1^4)$ for $(1,1,1,1)$, or $(1^3,2^2,1,3)$ for $(1,1,1,2,2,1,3)$.

To a composition $\alpha\vDash n$ we associate a set $\mathsf{set}(\alpha)=\mathsf{set}_n(\alpha)\subseteq [n-1]$ as follows:
\[\mathsf{set}(\alpha)=\{\alpha_1,\alpha_1+\alpha_2,\dots,\alpha_1+\cdots +\alpha_{k-1}\}.\]
Viceversa, to a subset $S\subseteq [n-1]$ whose elements are $i_1<i_2<\cdots <i_k$ we associate the composition \[\mathsf{comp}(S)=\mathsf{comp}_n(S)=(i_1,i_2-i_1,i_3-i_2,\dots,i_k-i_{k-1},n-i_k)\vDash n.\]
Notice that the functions $\mathsf{set}_n$ and $\mathsf{comp}_n$ are inverse of each others.

To a composition $\alpha\vDash n$, $\alpha=(\alpha_1,\alpha_2,\dots,\alpha_k)$ we associate its \emph{reversal} $\alpha^r=(\alpha_k,\alpha_{k-1},\dots,\alpha_1)$, its \emph{complement} $\alpha^c=\mathsf{comp}([n-1]\setminus \mathsf{set}(\alpha))$, and its \emph{transpose} $\alpha^t=(\alpha^r)^c=(\alpha^c)^r$. Notice that the reversal can also be described as $\alpha^r=\mathsf{comp}(n-\mathsf{set}(\alpha))$, where for every $S\subseteq [n-1]$, we denote $n-S:=\{n-i\mid i\in S\}\subseteq [n-1]$.

For example if $\alpha=(1,4,1,2)\vDash 8$, then $\alpha^r=(2,1,4,1)$, $\alpha^c=(2,1,1,3,1)$ and $\alpha^t=(1,3,1,1,2)$.

Given two compositions, we say that $\alpha$ is a \emph{refinement} of $\beta$ (equivalently $\beta$ is a \emph{coarsening} of $\alpha$), written $\alpha\preceq \beta$ (equivalently $\beta \succeq \alpha$), if we can obtain the parts of $\beta$ in order by adding adjacent parts of $\alpha$. For example $(1,\textcolor{red}{3},\textcolor{red}{1},1,\textcolor{blue}{2},\textcolor{blue}{1})\preceq (1,\textcolor{red}{4},1,\textcolor{blue}{3})$.

Observe that $\alpha\preceq \beta$ if and only if $\mathsf{set}(\beta)\subseteq \mathsf{set}(\alpha)$. Hence this partial order on compositions is a boolean lattice. So we can talk about the meet $\alpha\wedge\beta$: this is the coarsest common refinement of both $\alpha$ and $\beta$.  For example if $\alpha=(2,2,1)$ and $\beta=(3,2)$, then $\alpha\wedge \beta=(2,1,1,1)$. Notice that $\mathsf{set}(\alpha\wedge\beta)=\mathsf{set}(\alpha)\cup \mathsf{set}(\beta)$.

Given two compositions $\alpha=(\alpha_1,\alpha_2,\dots,\alpha_r)$ and $\beta=(\beta_1,\beta_2,\dots,\beta_s)$, we define the \emph{concatenation product}
\[\alpha\beta:=(\alpha_1,\alpha_2,\dots,\alpha_r,\beta_1,\beta_2,\dots,\beta_s).\]

Given two compositions $\alpha,\beta\vDash n$, let \[\gamma(\alpha,\beta):=(\gamma^1(\alpha,\beta),\gamma^2(\alpha,\beta),\dots,\gamma^{\ell(\beta)}(\alpha,\beta))\] 
be such that for every $i\in [\ell(\beta)]$
\[\gamma^i(\alpha,\beta)\vDash \beta_i \] 
and
\[ \gamma^1(\alpha,\beta)\gamma^2(\alpha,\beta)\cdots \gamma^{\ell(\beta)}(\alpha,\beta)=\alpha\wedge \beta.\]

For example, if again $\alpha=(2,2,1)$ and $\beta=(3,2)$, then $\alpha\wedge \beta=(2,1,1,1)$, $\gamma(\alpha,\beta)=(\gamma^1(\alpha,\beta),\gamma^2(\alpha,\beta))$ with $\gamma^1(\alpha,\beta)=(2,1)\vDash 3=\beta_1$ and $\gamma^2(\alpha,\beta)=(1,1)\vDash 2=\beta_2$.

Given a composition $\alpha\vDash n$, set
\[  \eta(\alpha):=\sum_{i=1}^{\ell(\alpha)-1}\sum_{j=1}^i\alpha_j\]
and
\[\eta(\gamma(\alpha,\beta)):=\sum_{i=1}^{\ell(\beta)}\eta(\gamma^i(\alpha,\beta)).\]

For example, if again $\alpha=(2,2,1)$ and $\beta=(3,2)$, then $\eta(\alpha)=2+(2+2)=6$ and $\eta(\gamma(\alpha,\beta))=\eta(\gamma^1(\alpha,\beta))+\eta(\gamma^2(\alpha,\beta))=2+1=3$. Observe that $\eta((n))=0$.

\smallskip

We denote by $\mathrm{QSym}$ the algebra of quasisymmetric functions in the variables $x_1,x_2,\dots$ and coefficients in $\mathbb{Q}(q)$, where $q$ is a variable. Given a composition $\alpha=(\alpha_1,\alpha_2,\dots,\alpha_k)$, we define the \emph{monomial quasisymmetric function} $M_\alpha$ as
\[M_{\alpha}=\sum_{i_1<i_2<\cdots <i_k}x_{i_1}^{\alpha_1}x_{i_2}^{\alpha_2}\cdots x_{i_k}^{\alpha_k}.\]

The set $\{M_\alpha\mid \alpha\text{ composition}\}$ is clearly a basis of $\mathrm{QSym}$ (this can also be taken as the definition of $\mathrm{QSym}$).

Given $n\in \mathbb{N}$ and $S\subseteq [n-1]$, we define the \emph{fundamental (Gessel) quasisymmetric function} $L_{n,S}$ as
\[L_{n,S}:=\mathop{\sum_{i_1\leq i_2\leq \cdots \leq i_n}}_{j\in S\Rightarrow i_j\neq i_{j+1}}x_{i_1}x_{i_2}\cdots x_{i_n}\] 
and for every $\alpha\vDash n$, we define $L_\alpha:=L_{n,\mathsf{set}(\alpha)}$. 

It is well known that 
$\{L_{\alpha}\mid \alpha\text{ composition}\}$ is a basis of $\mathrm{QSym}$.

We have the following three involutions of $\mathrm{QSym}$: $\psi: \mathrm{QSym}\to \mathrm{QSym}$, defined by $\psi(L_{\alpha}):=L_{\alpha^c}$, $\rho: \mathrm{QSym}\to \mathrm{QSym}$ defined by $\rho(L_{\alpha})=L_{\alpha^r}$, and $\omega: \mathrm{QSym}\to \mathrm{QSym}$ defined by $\omega(L_{\alpha})=L_{\alpha^t}$. Observe that these isomorphisms commute, and $\omega=\psi\circ \rho=\rho\circ \psi$. Moreover, it is easy to check that $\rho(M_\alpha)=M_{\alpha^r}$.
\begin{remark} \label{rmk:psi_restrict_to_omega}
Notice that these involutions restrict to the subalgebra of symmetric functions $\mathrm{Sym}\subset \mathrm{QSym}$, where $\rho$ restricts to the identity, while both $\psi$ and $\omega$ restrict to the involution $\omega$ in $\mathrm{Sym}$ such that $\omega(h_i)=e_i$ for every $i\geq 1$, where as usual $e_i$ and $h_i$ are the \emph{elementary} and the \emph{complete homogeneous symmetric functions} respectively (hence $h_n=L_{(n)}$ and $e_n=L_{(1^n)}$).
\end{remark}

\smallskip 

We will need to work with \emph{plethysm} of quasisymmetric functions. For a nice presentation of this fundamental tool we refer to \cite{Loehr-Remmel-plethystic-2011}. Here we limit ourselves to introduce the few things that are needed later in the text.

First of all, some notation: given a quasisymmetric function $f\in \mathrm{QSym}$, we denote its dependence on the \emph{alphabet} of variables $x_1,x_2,\dots$ by writing $f[X]$; it is useful to thing of $X$ as the series $x_1+x_2+\cdots$. Now given $f\in \mathrm{QSym}$, we want to define a quasisymmetric function in the alphabet of variables $\{x_iy_j\}_{i,j\geq 1}$, where $y_1,y_2,\dots$ is another alphabet of variables. 

If $f$ is symmetric, then there is an obvious way to do it: we consider simply $f(x_1y_1,x_2y_1,\dots,x_iy_j,\dots)$. Since $f$ is symmetric, it is not important in which order the $x_iy_j$ occur in this expression.

This leads to the plethystic notation $f[XY]$, where again it is useful to thing of $Y$ as the series $y_1+y_2+\cdots$, so that $XY$ is also the formal sum of the elements of the alphabet $\{x_iy_j\}_{i,j\geq 1}$. 

But if $f\in \mathrm{QSym}$ is not symmetric, then the order in which the $x_iy_j$ occur is relevant. Hence we have to make a choice. So for any composition $\alpha$ we define the plethysm $L_\alpha\left[XY\right]$ where the alphabet $XY=\sum_{i,j\geq 1}x_iy_j$ is considered in lexicographic order, i.e.
\[ x_1y_1<x_1y_2<x_1y_3<\cdots < x_2y_1<x_2y_2<x_2y_3<\cdots, \]
by setting
\[ L_\alpha[XY]:=\mathop{\sum_{a_1\leq a_2\leq \cdots \leq a_n}}_{a_j=a_{j+1}\Rightarrow j\notin \mathsf{set}(\alpha)}a_1a_2\cdots a_n \]
where each $a_i$ is an element of our alphabet $XY\equiv \{x_iy_j\}_{i,j\geq 1}$ and the inequalities $a_i\leq a_{i+1}$ are determined by the total order that we fixed on the alphabet.

Of course for any $f\in \mathrm{QSym}$ we define $f[XY]$ extending the above definition by linearity (since the $L_\alpha$ form a basis of $\mathrm{QSym}$).

In \cite{Loehr-Remmel-plethystic-2011} the authors define more generally the plethysm of a fundamental quasisymmetric function in a ``combinatorial alphabet'', which is in particular totally ordered, and the plethysm depends on the order. It turns out that our definition coincides with their definition in the present case. In particular, when applied to a symmetric function $f$ we recover the usual plethysm of symmetric functions (we do not recall the definition here, but it appears in \cite{Loehr-Remmel-plethystic-2011}).

\smallskip 

We observe that with the given order, the resulting series $L_\alpha[XY]$ is quasisymmetric separately in the alphabet $X$ and in the alphabet $Y$. The following proposition follows directly from the above definitions and notations.
\begin{proposition} \label{prop:L_Cauchy}
	Given a composition $\alpha\vDash n$, we have
	\begin{equation} \label{eq:L_Cauchy}
	L_\alpha[XY]=\sum_{\beta\vDash n}\prod_{i=1}^{\ell(\beta)}L_{\gamma^i(\alpha,\beta)}[Y] M_\beta[X]. 
	\end{equation}
\end{proposition}
\begin{remark}
	For example \[L_{(2,1)}[XY]=L_{(1)}^3[Y]M_{(1^3)}[X]+L_{(2)}[Y]L_{(1)}[Y]M_{(2,1)}[X]+L_{(1)}^3[Y]M_{(1,2)}[X]+L_{(2,1)}[Y]M_{(3)}[X]\]
	so the coefficient of $x_1^2x_2$ is $L_{(2)}[Y]L_{(1)}[Y]=h_{2}[Y]h_{1}[Y]$, which is symmetric in $Y$, while the coefficient of $y_1^2y_2$ is $3M_{(1^3)}[X]+2M_{(2,1)}[X]+3M_{(1,2)}[X]+M_{(3)}[X]$, which is not symmetric in $X$. This shows that we cannot switch the roles of $X$ and $Y$, hence our plethysm does indeed depend on the order that we choose on the alphabet.
	
	On the other hand, it is easy to check that in the special cases $\alpha=(n)$ and $\alpha=(1^n)$ we do recover the well-known Cauchy identities of $h_n[XY]$ and $e_n[XY]$ respectively. Hence our proposition can be seen as an extension of those Cauchy identities.
\end{remark}

To conclude this section, we compute a formula for the plethysm $L_\alpha\left[\frac{1}{1-q}\right]$.

\begin{lemma} \label{lem:L_ps}
	Given a composition $\alpha\vDash n$, consider the alphabet $\frac{1}{1-q}=1+q+q^2+\cdots$ (or better $\{q^i\mid i\in \mathbb{N}\}$) ordered naturally as $1=q^0<q^1<q^2<\cdots$. We have
	\begin{equation} \label{eq:L_plethysm}
	L_\alpha\left[\frac{1}{1-q}\right]=q^{\eta(\alpha^r)}h_{n}\left[\frac{1}{1-q}\right]. 
	\end{equation}
\end{lemma}
\begin{proof}
It is well known (see e.g.\ \cite{Stanley-Book-1999}*{Corollary~7.21.3}) that 
\begin{equation} \label{eq:hn_ps} h_n\left[\frac{1}{1-q}\right]=h(1,q,q^2,\dots)=\prod_{i=1}^n\frac{1}{1-q^i}, 
\end{equation} 
which $q$-counts the area between the interval $[0,n]$ of the $x$ axis and the infinite lattice paths starting from $(0,0)$, consisting of exactly $n$ unit east steps while all the other ones are unit north steps (this is just another way to picture the partitions with parts of sizes at most $n$). 

Then, to understand our formula, we can interpret the monomials occurring on the left hand side as $q$-counting the same area of the same paths, but with the extra condition that for every $i\in\mathsf{set}(\alpha)$ the $i$-th unit east step is necessarily followed by a unit north step. Then the factor $\eta(\alpha^r)$ takes simply into account the extra area added by the $\ell(\alpha)-1$ forced unit north steps. 
\end{proof}

\section{Colorings and (co)inversions}

Given $n\in \mathbb{Z}_{>0}$, let $G=([n],E)$ be a (simple) graph. 

A \emph{coloring} of $G$ is simply a function $\kappa:[n]\to \mathbb{Z}_{>0}$. We call $\mathsf{C}(G)$ the set of colorings of $G$. We can and will identify a coloring $\kappa\in \mathsf{C}(G)$ with the word $\kappa(1)\kappa(2)\cdots \kappa(n)$ in the alphabet $\mathbb{Z}_{>0}$, as well as with the element $(\kappa(1),\kappa(2),\cdots , \kappa(n))\in \mathbb{Z}_{>0}^n$, so that $\mathsf{C}(G)$ is identified with $\mathbb{Z}_{>0}^n$. Also, the elements of the symmetric group $\mathfrak{S}_n$ can be identified as elements of $\mathsf{C}(G)$ in the natural way.

A \emph{coloring} of $G$ is called \emph{proper} if $(i,j)\in E$ implies $\kappa(i)\neq \kappa(j)$. We call $\mathsf{PC}(G)$ the set of proper colorings of $G$. Notice that with the above identifications we always have $\mathfrak{S}_n\subseteq \mathsf{PC}(G)$.

Given $\kappa\in \mathsf{C}(G)$ a \emph{$G$-inversion} of $\kappa$ is a pair $(i,j)$ with $\{i,j\}\in E$, $i<j$ and $\kappa(i)>\kappa(j)$. Similarly, a  \emph{$G$-coinversion} of $\kappa$ is a pair $(i,j)$ with $\{i,j\}\in E$, $i<j$ and $\kappa(i)<\kappa(j)$. We denote by $\mathsf{Inv}_G(\kappa)$, respectively $\mathsf{CoInv}_G(\kappa)$, the set of $G$-inversions, respectively $G$-coinversions, of $\kappa$. Observe that with our notation we can identify both $\mathsf{Inv}_G(\kappa)$ and $\mathsf{CoInv}_G(\kappa)$ with (disjoint) subsets of $E$, and if $\kappa\in \mathsf{PC}(G)$, then one is the complement of the other in $E$.

We can now set for every $\kappa\in \mathsf{C}(G)$
\[ \mathsf{inv}_G(\kappa):=|\mathsf{Inv}_G(\kappa)|\quad \text{ and }\quad \mathsf{coinv}_G(\kappa):=|\mathsf{CoInv}_G(\kappa)|.\]
\begin{example} \label{ex:coinv_computation}
Consider the graph $G$ in Figure~\ref{fig:interval_graph}, and $\sigma=31852647\in \mathfrak{S}_8\subseteq \mathsf{PC}(G)$. Then 
\begin{align*}
	\mathsf{Inv}_G(\sigma) & =\{(1,2), (3,4),(3,5),(3,6),(6,7)\}\text{ and}\\
	\mathsf{CoInv}_G(\sigma)& =\{(1,3),(2,3),(2,4),(2,5),(2,6),(2,7),(5,6),(5,7),(6,8),(7,8)\},
\end{align*}
so that $\mathsf{inv}_G(\sigma)=5$ and $\mathsf{coinv}_G(\sigma)=10$.
\end{example}
\begin{remark} \label{rmk:inv_for_Kn}
Observe that in the case of the \emph{complete graph} $G=K_n=([n],E)$, given a permutation $\sigma\in\mathfrak{S}_n$ (which we think of as a proper coloring of $K_n$), $\mathsf{Inv}_{K_n}(\sigma)$ is just the set of the usual \emph{inversions} of $\sigma$, and $\mathsf{inv}_{K_n}(\sigma)$ is usually denoted simply by $\mathsf{inv}(\sigma)$.
\end{remark}

Let $\phi:\mathsf{C}(G)\to \mathfrak{S}_n$ be the \emph{standardization} from left to right: given the word $\kappa(1)\kappa(2)\cdots\kappa(n)$, if $c_1<c_2<\cdots<c_k$ is the ordered set of values $\kappa(i)$, then $\phi(\kappa)$ is the permutation obtained by replacing the $d_1$ occurrences of $c_1$ with the numbers $1,2,\dots,d_1$ from left to right, then the $d_2$ occurrences of $c_2$ with the numbers $d_1+1,d_1+2,\dots,d_1+d_2$ from left to right, and so on. For example $\phi(3253353)=216 347 5$.

\begin{remark} \label{rem:Inv_CoInv}
	Observe that for any $\kappa\in \mathsf{C}(G)$, 
	\[\mathsf{CoInv}_G(\kappa)\subseteq \mathsf{CoInv}_G(\phi(\kappa))\quad \text{and}\quad \mathsf{Inv}_G(\kappa)= \mathsf{Inv}_G(\phi(\kappa)).\]
	The asymmetry is due to the fact that the standardization $\phi$ is from left to right. But observe that if $\kappa\in \mathsf{PC}(G)$, then in fact $\mathsf{CoInv}_G(\kappa)= \mathsf{CoInv}_G(\phi(\kappa))$ as well.
\end{remark}

Recall that given a word $w=w_1w_2\cdots w_n$ in a totally ordered alphabet, the set of \emph{descents} of $w$ is the set
\[ \mathsf{Des}(w):=\{i\in [n-1]\mid w_i>w_{i+1} \}. \] 

While the next two lemmas are well known (they can be understood using Stanley's theory of $P$-partitions, see e.g\  \cite{Stanley-Book-1999}*{Chapter~7}), we prefer to provide here a short direct proof for completeness.

\begin{lemma}
	Given a simple graph $G=([n],E)$ and given $\sigma\in \mathfrak{S}_n$, the following statements about $\kappa\in \mathsf{C}(G)$ are equivalent:
	\begin{enumerate}
		\item  $\phi(\kappa)=\sigma$;
		\item $\kappa(\sigma^{-1}(i))\leq \kappa(\sigma^{-1}(i+1))$ for every $i\in [n-1]$ and\\
		$\kappa(\sigma^{-1}(i))< \kappa(\sigma^{-1}(i+1))$ for every $i\in \mathsf{Des}(\sigma^{-1})$.
	\end{enumerate}
\end{lemma}
\begin{proof}
	Suppose first (1), i.e.\ suppose that $\kappa\in \mathsf{C}(G)$ and $\phi(\kappa)=\sigma$.
	
	Observe that the standardization from left to right immediately implies that if $\sigma^{-1}(i)>\sigma^{-1}(i+1)$, then $\kappa(\sigma^{-1}(i+1))> \kappa(\sigma^{-1}(i))$. Similarly, if $\sigma^{-1}(i)<\sigma^{-1}(i+1)$, then $\kappa(\sigma^{-1}(i))\leq \kappa(\sigma^{-1}(i+1))$.
	
	Suppose now that $\kappa\in \mathsf{C}(G)$ satisfies (2).
	
	Let us show that $\phi(\kappa)=\sigma$: we want to show that for every $i,j\in [n]$ with $i<j$, $\kappa(i)\leq \kappa(j)$ if and only if $\sigma(i)<\sigma(j)$.
	
	If $\sigma(i)<\sigma(j)$, then using the first condition in (2) we have
	\[\kappa(i)=\kappa(\sigma^{-1}(\sigma(i)))\leq \kappa(\sigma^{-1}(\sigma(i)+1))\leq \kappa(\sigma^{-1}(\sigma(i)+2))\leq\cdots \leq \kappa(\sigma^{-1}(\sigma(j)))=\kappa(j).\]
	
	Viceversa, if $\sigma(i)>\sigma(j)$, then using the first condition in (2) we have
	\[\kappa(j)=\kappa(\sigma^{-1}(\sigma(j)))\leq \kappa(\sigma^{-1}(\sigma(j)+1))\leq \kappa(\sigma^{-1}(\sigma(j)+2))\leq\cdots \leq \kappa(\sigma^{-1}(\sigma(i)))=\kappa(i),\]
	and if there was the equality $\kappa(j)=\kappa(i)$, then using the second condition in (2) we must have
	\[j=\sigma^{-1}(\sigma(j))<\sigma^{-1}(\sigma(j)+1)<\sigma^{-1}(\sigma(j)+2)<\cdots <\sigma^{-1}(\sigma(i))=i\]
	contradicting $i<j$. This completes the proof of $\phi(\kappa)=\sigma$, and therefore also the proof of the lemma.
\end{proof}
For any $\kappa\in \mathsf{C}(G)$, set
\[x_\kappa:=\prod_{i=1}^nx_{\kappa(i)},\]
where $x_1,x_2,\dots$ are variables.

An immediate corollary of the previous lemma is the following result.
\begin{lemma}
	Given a simple graph $G=([n],E)$ and given $\sigma\in \mathfrak{S}_n$, we have
	\[\mathop{\sum_{\kappa\in \mathsf{C}(G)}}_{\phi(\kappa)=\sigma}x_\kappa=L_{n,\mathsf{Des}(\sigma^{-1})}.\]	
\end{lemma}

Let us denote by $\mathsf{Inv}(G)$ the (finite) set of possible sets of $G$-inversions of a coloring of $G$: in other words
\[\mathsf{Inv}(G):=\{\mathsf{Inv}_G(\sigma)\mid \sigma\in \mathfrak{S}_n\}.\]
Similarly, set
\[\mathsf{CoInv}(G):=\{\mathsf{CoInv}_G(\sigma)\mid \sigma\in \mathfrak{S}_n\}.\]

Combining the results and remarks of this section, we get the following formula.
\begin{proposition}\label{prop:invG_LLT}
	Given a simple graph $G=([n],E)$, for every $S\in \mathsf{Inv}(G)$ we have
	\[\mathop{\sum_{\kappa\in \mathsf{C}(G)}}_{\mathsf{Inv}_G(\kappa)=S}q^{\mathsf{inv}_G(\kappa)}x_\kappa=\mathop{\sum_{\sigma\in \mathfrak{S}_n}}_{\mathsf{Inv}_G(\sigma)=S}q^{\mathsf{inv}_G(\sigma)}L_{n,\mathsf{Des}(\sigma^{-1})}=q^{|S|}\mathop{\sum_{\sigma\in \mathfrak{S}_n}}_{\mathsf{Inv}_G(\sigma)=S}L_{n,\mathsf{Des}(\sigma^{-1})}.\]
\end{proposition}
Notice that, because of Remark~\ref{rem:Inv_CoInv}, a similar formula with $\mathsf{Inv}$ and $\mathsf{inv}$ replaced by $\mathsf{CoInv}$ and $\mathsf{coinv}$ is not supposed to hold in general.

\section{Interval graphs and colorings}

Given $n\in \mathbb{Z}_{>0}$, let $G=([n],E)$ be an interval graph, i.e.\ $G\in \mathcal{IG}_n$.

Given $\tau \in \mathfrak{S}_n$, set \[\mathsf{Des}_G(\tau):=\{i\in [n-1]\mid \tau(i)>\tau(i+1)\text{ or } \{\tau(i),\tau(i+1)\}\in E(G)\} \subseteq [n-1]. \]

While the next two lemmas are implicit in the work of Shareshian and Wachs \cite{Shareshian_Wachs_Advances}, we prefer to provide here a short direct proof, appearing in unpublished work of Philippe Nadeau and Vasu Tewari, for completeness. We will relate them to the results in \cite{Shareshian_Wachs_Advances} in Section~\ref{sec:chromatic_LLT}. 

\begin{lemma}
	Given $G=([n],E)$ an interval graph and given $\sigma\in \mathfrak{S}_n$, the following statements about $\kappa\in \mathsf{C}(G)$ are equivalent:
	\begin{enumerate}
		\item $\kappa\in \mathsf{PC}(G)$ and $\phi(\kappa)=\sigma$;
		\item $\kappa(\sigma^{-1}(i))\leq \kappa(\sigma^{-1}(i+1))$ for every $i\in [n-1]$ and\\
		$\kappa(\sigma^{-1}(i))< \kappa(\sigma^{-1}(i+1))$ for every $i\in \mathsf{Des}_G(\sigma^{-1})$.
	\end{enumerate}
\end{lemma}
\begin{proof}
	Suppose first (1), i.e.\ suppose that $\kappa\in \mathsf{PC}(G)$ and $\phi(\kappa)=\sigma$.
	
	Observe that the standardization from left to right immediately implies that if $\sigma^{-1}(i)>\sigma^{-1}(i+1)$, then $\kappa(\sigma^{-1}(i+1))> \kappa(\sigma^{-1}(i))$. Similarly, if $\sigma^{-1}(i)<\sigma^{-1}(i+1)$, then $\kappa(\sigma^{-1}(i))\leq \kappa(\sigma^{-1}(i+1))$. But if $\{\sigma^{-1}(i),\sigma^{-1}(i+1)\}\in E$, then the last inequality must be strict since $\kappa\in \mathsf{PC}(G)$.
	
	Suppose now that $\kappa\in \mathsf{C}(G)$ satisfies (2).
	
	Let us show that $\kappa\in \mathsf{PC}(G)$: given $i,j\in [n]$ with $i<j$ and $\{\sigma^{-1}(i),\sigma^{-1}(i)\}\in E$, we want to show that  $\kappa(\sigma^{-1}(i))\neq \kappa(\sigma^{-1}(j))$. By iterating the first inequality of (2) we must have 
	\[\kappa(\sigma^{-1}(i))\leq \kappa(\sigma^{-1}(i+1))\leq\cdots \leq  \kappa(\sigma^{-1}(j)).\] If by contradiction we had the equality $\kappa(\sigma^{-1}(i))=\kappa(\sigma^{-1}(j))$, then the second condition in (2) would imply that \[\sigma^{-1}(i)<\sigma^{-1}(i+1)<\sigma^{-1}(i+2)<\cdots <\sigma^{-1}(j)  \]
	and that $\{\sigma^{-1}(i+r),\sigma^{-1}(i+r+1)\}\notin E$ for every $r=0,1,\dots,j-i-1$. But since our $G$ is an interval graph, $\{\sigma^{-1}(i),\sigma^{-1}(j)\}\in E$ implies $\{\sigma^{-1}(i),\sigma^{-1}(i+1)\}\in E$, which gives a contradiction. Hence we must have $\kappa(\sigma^{-1}(i))<\kappa(\sigma^{-1}(j))$, completing the proof that $\kappa$ is in $\mathsf{PC}(G)$.
	
	Now let us show that $\phi(\kappa)=\sigma$: we want to show that for every $i,j\in [n]$ with $i<j$, $\kappa(i)\leq \kappa(j)$ if and only if $\sigma(i)<\sigma(j)$.
	
	If $\sigma(i)<\sigma(j)$, then using the first condition in (2) we have
	\[\kappa(i)=\kappa(\sigma^{-1}(\sigma(i)))\leq \kappa(\sigma^{-1}(\sigma(i)+1))\leq \kappa(\sigma^{-1}(\sigma(i)+2))\leq\cdots \leq \kappa(\sigma^{-1}(\sigma(j)))=\kappa(j).\]
	
	Viceversa, if $\sigma(i)>\sigma(j)$, then using the first condition in (2) we have
	\[\kappa(j)=\kappa(\sigma^{-1}(\sigma(j)))\leq \kappa(\sigma^{-1}(\sigma(j)+1))\leq \kappa(\sigma^{-1}(\sigma(j)+2))\leq\cdots \leq \kappa(\sigma^{-1}(\sigma(i)))=\kappa(i),\]
	and if there was the equality $\kappa(j)=\kappa(i)$, then using the second condition in (2) we must have
	\[j=\sigma^{-1}(\sigma(j))<\sigma^{-1}(\sigma(j)+1)<\sigma^{-1}(\sigma(j)+2)<\cdots <\sigma^{-1}(\sigma(i))=i\]
	contradicting $i<j$. This completes the proof of $\phi(\kappa)=\sigma$, and therefore also the proof of the lemma.
\end{proof}
An immediate corollary of the previous lemma is the following result.
\begin{lemma}
	Given $G=([n],E)$ an interval graph and given $\sigma\in \mathfrak{S}_n$, we have
	\[\mathop{\sum_{\kappa\in \mathsf{PC}(G)}}_{\phi(\kappa)=\sigma}x_\kappa=L_{n,\mathsf{Des}_G(\sigma^{-1})}.\]	
\end{lemma}

Combining the previous lemmas with Remark~\ref{rem:Inv_CoInv}, we get the following formulas.
\begin{proposition} \label{prop:inv_coinv_formulae}
	Given $G=([n],E)$ an interval graph, for every $S\in \mathsf{Inv}(G)$ we have
	\[\mathop{\sum_{\kappa\in \mathsf{PC}(G)}}_{\mathsf{Inv}_G(\kappa)=S}q^{\mathsf{inv}_G(\kappa)}x_\kappa=\mathop{\sum_{\sigma\in \mathfrak{S}_n}}_{\mathsf{Inv}_G(\sigma)=S}q^{\mathsf{inv}_G(\sigma)}L_{n,\mathsf{Des}_G(\sigma^{-1})}=q^{|S|}\mathop{\sum_{\sigma\in \mathfrak{S}_n}}_{\mathsf{Inv}_G(\sigma)=S}L_{n,\mathsf{Des}_G(\sigma^{-1})},\]
	and for every $S\in \mathsf{CoInv}(G)$ we have
	\[\mathop{\sum_{\kappa\in \mathsf{PC}(G)}}_{\mathsf{CoInv}_G(\kappa)=S}q^{\mathsf{coinv}_G(\kappa)}x_\kappa=\mathop{\sum_{\sigma\in \mathfrak{S}_n}}_{\mathsf{CoInv}_G(\sigma)=S}q^{\mathsf{coinv}_G(\sigma)}L_{n,\mathsf{Des}_G(\sigma^{-1})}=q^{|S|}\mathop{\sum_{\sigma\in \mathfrak{S}_n}}_{\mathsf{CoInv}_G(\sigma)=S}L_{n,\mathsf{Des}_G(\sigma^{-1})}.\]
\end{proposition}

\section{Increasing spanning forests and quasisymmetric functions}

Given a graph $G=([n],E)$, we say that a subgraph $F\subseteq G$ is a \emph{spanning forest} if $F$ is a forest on the vertices $[n]$. In this case, the connected components are labelled trees, with the vertex set contained in $[n]$. Given such a tree $T$, we call $root(T)$ its minimal vertex. Then $T$ is called \emph{increasing} if in the paths stemming from $root(T)$ the other vertices appear in increasing order. 

A spanning forest $F$ of a graph $G=([n],E)$ is called \emph{increasing} if all its connected components are increasing trees. In this case, we think of $F$ as the ordered collection $F=(T_1,T_2,\dots,T_k)$, where the $T_i$ are its connected components, ordered so that
\[
root(T_1)<root(T_2)<\cdots <root(T_k).
\]

For example, the forest $F=(T_1,T_2)$ in Figure~\ref{fig:isf_example}, where $T_1=(V(T_1),E(T_1))=(\{1,3\},\{(1,3)\})$ and $T_2=(V(T_2),E(T_2))=(\{2,4,5,6,7,8\},\{(2,4),(2,5),(2,6),(5,7),(6,8)\})$, is an increasing spanning forest of the graph $G$ in Figure~\ref{fig:interval_graph}, .

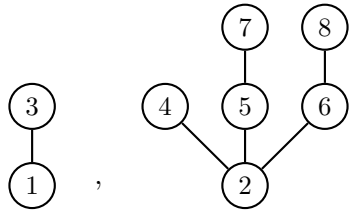
\begin{figure}[ht]
	\centering
	\begin{tikzpicture}[auto,node distance=3cm,
		thick,root node/.style={shape=circle, inner sep=3pt, fill=red,draw},main node/.style={shape=circle, inner sep=3pt, fill=black,draw},external node/.style={shape=circle, inner sep=6pt,draw},scale=.35]
		
		\node[external node] (1) at (0,0) {};
		\node[external node] (2) at (8,0) {};
		\node[external node] (3) at (0,3) {};
		\node[external node] (4) at (5,3) {};
		\node[external node] (5) at (8,3) {};
		\node[external node] (6) at (11,3) {};
		\node[external node] (7) at (8,6) {};
		\node[external node] (8) at (11,6) {};
		
		\node at (0,3) {$3$};
		\node at (0,0) {$1$};
		\node at (8,0) {$2$};
		\node at (5,3) {$4$};
		\node at (8,3) {$5$};
		\node at (11,3) {$6$};
		\node at (8,6) {$7$};
		\node at (11,6) {$8$};
		\node at (2.5,0) {\large{$,$}};

		\path
		(1) edge node {} (3)
		(2) edge node {} (4)
		(2) edge node {} (5)
		(2) edge node {} (6)
		(5) edge node {} (7)
		(6) edge node {} (8);
	\end{tikzpicture}
	\caption{An example of increasing spanning forest of the graph $G$ in Figure~\ref{fig:interval_graph}. }
	\label{fig:isf_example}
\end{figure}

 We denote by $\mathsf{ISF}(G)$ the set of increasing spanning forests of $G$.
 
 For example, for $G=([3],\{(1,2),(1,3)\})$, the increasing spanning forests of $G$ are $(([3],\{(1,2),(1,3)\}))$, $((\{1,3\},\{(1,3)\}),(\{2\},\varnothing))$, $((\{1,2\},\{(1,2)\}),(\{3\},\varnothing))$ and $((\{1\},\varnothing),(\{2\},\varnothing),(\{3\},\varnothing))$: see Figure~\ref{fig:isf_small_example}.
 
 \begin{figure}[ht]
 	\centering
 	\begin{tikzpicture}[auto,node distance=3cm,
 	thick,root node/.style={shape=circle, inner sep=3pt, fill=red,draw},main node/.style={shape=circle, inner sep=3pt, fill=black,draw},external node/.style={shape=circle, inner sep=6pt,draw},scale=.35]

 	\node[external node] (11) at (-1.5,0) {}; 	
 	\node[external node] (12) at (-3,3) {};
 	\node[external node] (13) at (0,3) {};

 	\node[external node] (21) at (5,0) {};
 	\node[external node] (22) at (5,3) {};
 	\node[external node] (23) at (8,0) {};

 	\node[external node] (31) at (13,0) {};
 	\node[external node] (32) at (13,3) {};
 	\node[external node] (33) at (16,0) {};
 	
 	\node[external node] (41) at (21,0) {};
 	\node[external node] (42) at (24,0) {};
 	\node[external node] (43) at (27,0) {};
  	\node at (-1.5,0) {$1$};	
 	\node at (-3,3) {$2$};
 	\node at (0,3) {$3$};
 	\node at (2.5,0) {\large{$;$}};
 	
 	\node at (5,0) {$1$};
 	\node at (5,3) {$3$};
 	\node at (6.5,0) {\large{$,$}};
 	\node at (8,0) {$2$};
 	\node at (10.5,0) {\large{$;$}};
 	
 	\node at (13,0) {$1$};
 	\node at (13,3) {$2$};
 	\node at (14.5,0) {\large{$,$}};
 	\node at (16,0) {$3$};
 	\node at (18.5,0) {\large{$;$}};
 	
 	\node at (21,0) {$1$};
 	\node at (22.5,0) {\large{$,$}};
 	\node at (24,0) {$2$};
 	\node at (25.5,0) {\large{$,$}};
 	\node at (27,0) {$3$};

 	\path
 	(11) edge node {} (12)
 	(11) edge node {} (13)
 	(21) edge node {} (22)
 	(31) edge node {} (32);
 	\end{tikzpicture}
 	\caption{The increasing spanning forests of the graph $G=([3],\{(1,2),(1,3)\})$. }
 	\label{fig:isf_small_example}
 \end{figure}
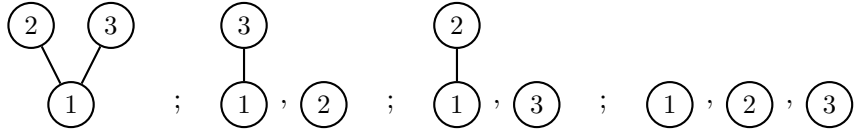

 \smallskip

Given a (simple) graph $G=([n],E)$ and an $F\in \mathsf{ISF}(G)$, $F=(T_1,T_2,\dots,T_k)$, we say that a pair $(u,v)$ with $u,v\in [n]$ is a \emph{$G$-inversion} of $F$ if $u\in V(T_i)$, $v\in V(T_j)$, $i>j$ and $(u,v)\in E$ (so that $u<v$). Given an edge $(u,v)\in E(T_i)$ of $T_i$ we define its \emph{weight} in $G$, denoted $\mathsf{wt}_G((u,v))$, to be the number of $w\in V(T_i)$ vertex of $T_i$ such that $u\leq w<v$ and $(w,v)\in E(G)$. So for every tree $T_i$ we define its \emph{weight} in $G$ as
\[ \mathsf{wt}_G(T_i)=\sum_{(u,v)\in E(T_i)}\mathsf{wt}_G((u,v))\]
and finally the \emph{weight} of $F$ (in $G$) as
\[\mathsf{wt}_G(F):=\# \{G\text{-inversions of }F \}+\sum_{i=1}^k \mathsf{wt}_G(T_i).\]

\begin{example}\label{ex:wtG_computation}
The forest $F=(T_1,T_2)$ in Figure~\ref{fig:isf_example} is an increasing spanning forest of the graph $G$ in Figure~\ref{fig:interval_graph}: we observe that its only $G$-inversion is $(2,3)$ (as $3$ occurs in $T_1$, $2$ occurs in $T_2$ and $(2,3)\in E$), $\mathsf{wt}_G(T_1)=\mathsf{wt}_G((1,3))=1$, and \begin{align*}
\mathsf{wt}_G(T_2)& =\mathsf{wt}_G((2,4))+\mathsf{wt}_G((2,5))+\mathsf{wt}_G((2,6))+\mathsf{wt}_G((5,7))+\mathsf{wt}_G((6,8))\\
& = 1+1+2+2+2=8,
\end{align*}
so that $\mathsf{wt}_G(F)=1+1+8=10$. 
\end{example}

\begin{remark}
Our weight function $\mathsf{wt}_G$ is almost the same as the one appearing in \cite{Abreu_Nigro_Forests}: the difference of the two is the number of edges in the forest.
\end{remark}

Let $G=([n],E)$ be an interval graph, i.e.\ $G\in \mathcal{IG}_n$. We are going to define a function $\Phi_G:\mathsf{PC}(G)\to \mathsf{ISF}(G)$ via an algorithm. But first we need an auxiliary function, that we call $\mathrm{getW}$ and we define with the following algorithm. 

\begin{algorithm}
	\caption{Algorithm defining the function $\mathrm{getW}(G,v,S,\kappa)$}\label{alg:getW}
	\begin{algorithmic}
		\Require A graph $G=([n],E)$, $S\subset [n]$, $v\in [n]\setminus S$, and $\kappa\in \mathsf{PC}(G)$
		\Ensure $W$\Comment{It will be $W\subseteq S\cup\{v\}$}
		\State $W \gets \{v\}$
		\For{$w\in S$}
		\If{$\{u\in W\mid u<w, (u,w)\in E\text{ and }\kappa(u)<\kappa(w)\}\neq \varnothing$}
		\State $W\gets W\cup\{w\}$
		\EndIf
		\EndFor
	\end{algorithmic}
\end{algorithm}

Now we are ready to define the function $\Phi_G$: Algorithm~\ref{alg:main} will construct the image $\Phi_G(\kappa)\in \mathsf{IFS}(G)$, $\Phi_G(\kappa)=(T_1,T_2,\dots)$, one tree at the time, in the order $T_1,T_2,\dots$ with $root(T_1)<root(T_2)<\cdots$.

\begin{algorithm}
	\caption{The algorithm defining the function $\Phi_G(\kappa)$}\label{alg:main}
	\begin{algorithmic}
		\Require A graph $G=([n],E)$ and $\kappa\in \mathsf{PC}(G)$
		\Ensure $F=(T_1,T_1,\dots)$ \Comment{It will be $F\in \mathsf{ISF}(G)$}
		\State $S \gets [n]$
		\State $F \gets (\ )$ \Comment{Empty list}
		\While{$S \neq \varnothing$}
		\State $v\gets \min(S)$
		\State $S\gets S\setminus\{v\}$
		\State $T=(V(T),E(T))\gets (\{v\},\varnothing)$ \Comment{The tree we are going to build}
		\State $W\gets \mathrm{getW}(G,v,S,\kappa)$ \Comment{Defined in Algorithm \ref{alg:getW}}
		\For{$i\in \{2,\dots,\# W \}$}
		\State $L \gets \{u\in V(T)\mid u<W_i\text{ and }(u,W_i)\in E\}$ \Comment{$W=\{W_1<W_2<\cdots <W_{\# W}\}$}
		\State $r \gets \# \{u\in L\mid (u,W_i)\in E\text{ and }\kappa(u)<\kappa(W_i)\}$
		\State $T \gets (V(T)\cup\{W_i\},E(T)\cup\{(L_{\# L-r+1},W_i)\})$   \Comment{$L=\{L_1<L_2<\cdots <L_{\# L}\}$}
		\State $S\gets S\setminus\{W_i\}$
		\EndFor
		\State Append $T$ to the right of $F$
		\EndWhile
	\end{algorithmic}
\end{algorithm}
Before discussing Algorithm~\ref{alg:main}, we illustrate it with an example.
\begin{example}
Let $G$ be the (interval) graph in Figure~\ref{fig:interval_graph}, and $\sigma=31852647\in \mathfrak{S}_8\subseteq \mathsf{PC}(G)$. 

We enter the first iteration of the \textbf{while} loop of Algorithm~\ref{alg:main} with $S=[8]$, we set $v=1$, $T=(\{1\},\varnothing)$, and the function $\mathrm{getW}$ defined by Algorithm~\ref{alg:getW} gives us $W=\{1,3\}$. So in the \textbf{for} loop we only have $i=2$ to work out: here $W_2=3$, $L=\{1\}$, $r=1$, and we update $T$, getting $T=(\{1,3\},\{(1,3)\})$, which becomes $T_1$ of our output $F=(T_1,\dots)$. Along the way of the first iteration of the \textbf{while} loop we updated $S$ ending up with $S=\{2,4,5,6,7,8\}$, which are the vertices that we still need to add to $F$.

We enter the second iteration of the \textbf{while} loop of Algorithm~\ref{alg:main} with $S=\{2,4,5,6,7,8\}$, we set $v=2$, $T=(\{2\},\varnothing)$, and the function $\mathrm{getW}$ defined by Algorithm~\ref{alg:getW} gives us $W=\{2,4,5,6,7,8\}$. Hence we enter the \textbf{for} loop: what happens is described below in Table~\ref{table1}.
\begin{center}
\begin{table}
\begin{tabular}{c|c|c|c|c|c}
	i & $W_i$ & L & r & T\\
	\hline 2 & 4 & \{2\} & 1 & (\{2,4\},\{(2,4)\})\\
	\hline 3 & 5 & \{2\} & 1 & (\{2,4,5\},\{(2,4),(2,5)\})\\
	\hline 4 & 6 & \{2,5\} & 2 & (\{2,4,5,6\},\{(2,4),(2,5),(2,6)\})\\
	\hline 5 & 7 & \{2,5,6\} & 2 & (\{2,4,5,6,7\},\{(2,4),(2,5),(2,6),(5,7)\})\\
	\hline 6 & 8 &\{6,7\} & 2 & (\{2,4,5,6,7,8\},\{(2,4),(2,5),(2,6),(5,7),(6,8)\})
\end{tabular}
\caption{}\label{table1}
\end{table}
\end{center}
Hence at the end we computed $T_2=(\{2,4,5,6,7,8\},\{(2,4),(2,5),(2,6),(5,7),(6,8)\})$. At this point $S$ is empty, hence the algorithm terminates, giving us $F=(T_1,T_2)$: this forest is depicted in Figure~\ref{fig:isf_example}. Notice that we already computed $\mathsf{coinv}_G(\sigma)=10$ in Example~\ref{ex:coinv_computation} and $\mathsf{wt}_G(F)=10$ in Example~\ref{ex:wtG_computation}: we will see shortly that this equality is not a coincidence. 
\end{example}
\begin{proposition} \label{prop:PhiG_welldefined}
Given a graph $G=([n],E)$, the Algorithm~\ref{alg:main} defines a function $\Phi_G:\mathsf{PC}(G)\to \mathsf{ISF}(G)$.
\end{proposition}
\begin{proof}
First of all notice that the function $\Phi_G$ is well defined: clearly the algorithm terminates, producing a forest with vertex set $[n]$. We need to check that the edges of $\Phi_G(\kappa)$ are in $G$, and finally that $\Phi_G(\kappa)$ is increasing.

Notice that the $j$-th iteration of the \textbf{while} loop produces the tree $T_j$ in $\Phi_G(\kappa)$. The call of the function $\mathrm{getW}$ in the $j$-th iteration of the \textbf{while} loop produces the vertex set $W=V(T_j)$ of $T_j$. The $j$-th iteration of the \textbf{while} loop adds the vertex $root(T_j)$ before the \textbf{for} loop, and then the \textbf{for} loop adds to $T_j$ the vertices of $W$ in increasing order, in such a way that $W_i$ gets added in $T_j$ with an edge $(u,W_i)\in E$ with $u<W_i$, since $L$ is precisely a set of elements of this type. All this shows that indeed $\Phi_G(\kappa)\in \mathsf{ISF}(G)$. 
\end{proof}

The first nontrivial property of the function $\Phi_G$ is its surjectivity. Indeed we will define a function  $f_G:\mathsf{ISF}(G)\to \mathfrak{S}_n\subset \mathsf{PC}(G)$ such that $\Phi_G \circ f_G(F)=F$ for every $F\in \mathsf{ISF}(G)$. 

We start by defining a function $\mathrm{colortree}(G,T,d)$ that takes as input a graph $G=([n],E)$, an increasing subtree $T$ of $G$ of size $k\leq n$, and an integer $d$, and it returns a permutation $\sigma$ of the numbers $\{d,d+1,\dots,d+k-1\}$: this function is defined in Algorithm~\ref{alg:PhiG_inverse_tree}. We will show that if $T$ is an increasing spanning tree of $G$, then indeed $\Phi_G(\mathrm{colortree}(G,T,d))=F:=(T)$.

\begin{algorithm}
	\caption{The algorithm defining $\mathrm{colortree}(G,T,d)$}\label{alg:PhiG_inverse_tree}
	\begin{algorithmic}
		\Require A graph $G=([n],E)$, an increasing subtree $T$ of $G$ of size $k\leq n$, and an integer $d$
		\Ensure $\sigma$ permutation of $\{d,d+1,\dots,d+k-1\}$ 
		\State $\sigma_1\gets d$
		\State $\sigma \gets (\sigma_1)$ \Comment{$\sigma$ will be a vector}
		\State $E(T) \gets ((a_1,b_1),(a_2,b_2),\ldots)$ edges of $T$ sorted by the $b_i$'s \Comment{$V(T)=\{root(T)<b_1<b_2<\cdots\}$}
		\For{$i$ from $1$ to $|E(T)|$} \Comment{$|E(T)|=|V(T)|-1=k-1$}
		\State  $L \gets \{u\in V(T)\mid u<b_i\text{ and }(u,b_i)\in E\}$ \Comment{$a_i\in L$}
		\State $r\gets \#\{u\in L\mid u> a_i\}$  
		\State $cL\gets (\sigma_{i_1},\sigma_{i_2},\cdots)$  where  $L=\{i_1<i_2<\cdots \}$
		\State $\sigma_{i+1} \gets \sigma_s+1$ where $s$ is such that $r=\#\{j \mid \sigma_{i_j}<\sigma_s\}$ 
		\For{$j\in [i]$}
		\If{$\sigma_j\geq \sigma_{i+1}$}
		\State $\sigma_j\gets \sigma_j+1$
		
		\EndIf
		\EndFor
		\State $\sigma \gets (\sigma_1,\sigma_2,\dots,\sigma_i,\sigma_{i+1})$ 
		\EndFor
	\end{algorithmic}
\end{algorithm}

Before discussing the algorithm, let us work out an example.
\begin{example}
Let $G=([4],\{(1,2),(1,3),(1,4),(2,3),(2,4)\})$, $T=([4],\{(1,2),(1,4),(2,3)\})$ and $d=2$. In Algorithm~\ref{alg:PhiG_inverse_tree} first of all we set $\sigma_1=2$, $\sigma=(\sigma_1)=(2)$, and $E(T)=((1,2),(2,3),(1,4))$, i.e.\ we get the edges of $T$ in increasing order with respect to the second coordinate (so $b_1=2$, $b_2=3$ and $b_3=4$, and $a_1=a_3=1$, $a_2=2$). Then in the first iteration of the first \textbf{for} loop ($i=1$) we get $L=\{1\}$, $r=0$, $cL=(2)$, $\sigma_2=\sigma_1+1=3$, and now in the internal \textbf{for} loop we do nothing ($\sigma_1<\sigma_2$), hence we finally update $\sigma$ to $\sigma=(\sigma_1,\sigma_2)=(2,3)$.

In the second iteration of the first \textbf{for} loop ($i=2$) we get $L=\{1,2\}$, $r=0$, $cL=(2,3)$, $\sigma_3=\sigma_1+1=3$, and now in the internal \textbf{for} loop we update $\sigma_2=4$ ($3=\sigma_2\geq \sigma_3=3$), hence we finally update $\sigma$ to $\sigma=(\sigma_1,\sigma_2,\sigma_3)=(2,4,3)$.

In the third and last iteration of the first \textbf{for} loop ($i=3$) we get $L=\{1,2\}$, $r=1$, $cL=(2,4)$, $\sigma_4=\sigma_2+1=5$, and now in the internal \textbf{for} loop we do nothing ($\sigma_k<\sigma_4$ for every $k\in [3]$), hence we finally update $\sigma$ to $\sigma=(\sigma_1,\sigma_2,\sigma_3,\sigma_4)=(2,4,3,5)$. It can be checked that $\Phi_G(\sigma)=F=(T)$.
\end{example}

\begin{proposition} \label{prop:PhiG_inverse_tree}
Algorithm~\ref{alg:PhiG_inverse_tree} defines a function $\mathrm{colortree}(G,T,d)$ that takes as input a graph $G=([n],E)$, an increasing subtree $T$ of $G$ of size $k\leq n$, and an integer $d$, and it returns a permutation $\sigma$ of the numbers $\{d,d+1,\dots,d+k-1\}$. Moreover, if $T$ is an increasing spanning tree of $G$, then $\Phi_G(\mathrm{colortree}(G,T,d))=F:=(T)$.
\end{proposition}
\begin{proof}
It is easy to see that Algorithm~\ref{alg:PhiG_inverse_tree} is well defined and that it always terminates. Indeed, notice that since $T$ is an increasing tree, the edges $E(T)$ are such that the second coordinates are all distinct, and they give all the vertices of $T$ except $root(T)$ (each one appearing exactly once). It is also clear that the output is a permutation of $\{d,d+1,\dots,d+k-1\}$: we start with $\sigma=(d)$, and then we always add $\sigma_{i+1}$ by taking a $\sigma_s$ that we already have, augment it by $1$, and then in the inner \textbf{for} loop we increase by $1$ all the $\sigma_j$ that are $\geq \sigma_{i+1}$ for $j\in [i]$, hence at the end of the $i$-th iteration of the first \textbf{for} loop our $\sigma$ will necessarily contain the numbers $\{d,d+1,\dots,d+i\}$.

To see why $\Phi_G$ should send our $\sigma$ into the forest $F=(T)\in \mathsf{ISF}(G)$ containing the only increasing tree $T$, we observe that when we added $\sigma_{i+1}$ to our coloring (in the $i$-th iteration of the first \textbf{for} loop) we made sure that the number of $u\in V(T)$ smaller than the $(i+1)$-th vertex of $T$, i.e.\ $b_i$ in our notation, that are connected in $G$ and that they have a smaller corresponding color, i.e.\ value of $\sigma$, is precisely the number that in Algorithm~\ref{alg:main} would lead to join the vertex $b_i$ to the vertex $a_i$. The rescaling occurring in the inner \textbf{for} loop is only there to guarantee that we get a proper coloring, in fact a permutation of distinct colors, as it does not affect the relative order of the previous colors, hence it does not change the outcome of Algorithm~\ref{alg:main}.
\end{proof}

We are now in a position to define a function  $f_G:\mathsf{ISF}(G)\to \mathfrak{S}_n\subset \mathsf{PC}(G)$ such that $\Phi_G \circ f_G(F)=F$ for every $F\in \mathsf{ISF}(G)$. We will define $f_G$ via Algorithm~\ref{alg:PhiG_inverse}.

\begin{algorithm}
	\caption{The algorithm defining the section $f_G:\mathsf{ISF}(G)$ of $\Phi_G$}\label{alg:PhiG_inverse}
	\begin{algorithmic}
		\Require A graph $G=([n],E)$ and $F=(T_1,T_2,\dots,T_k)\in \mathsf{ISF}(G)$
		\Ensure $\sigma\in \mathfrak{S}_n$ \Comment{It will be $\Phi_G(\sigma)=F$}
		\State $clrT\gets ()$ \Comment{empty vector}
		\For{$i$ from $1$ to $k$}
		\State $d\gets n+1-\sum_{j=1}^{i-1}|V(T_j)|$
		\State $clrT_i\gets \mathrm{colortree}(G,T_i,d)$
		\State Append $clrT_i$ to the right of $clrT$ 
		\EndFor
		\State $\sigma\gets ()$ \Comment{empty vector}
		\For{$i$ from $1$ to $n$}
		\For{$j$ from $1$ to $k$}
		\If{$i\in V(T_j)$}
		\State $r=\#\{k\in V(T_j)\mid k\leq i\}$
		\State Append $c_r$ to the right of $\sigma$, where $clrT_j=\{c_1<c_2<\cdots \}$
		\EndIf
		\EndFor
		\EndFor
	\end{algorithmic}
\end{algorithm}

\begin{theorem} \label{thm:PhiG_surjective}
	Let $G=([n],E)$ a graph. Then Algorithm~\ref{alg:PhiG_inverse} defines a function $f_G:\mathsf{ISF}(G)\to \mathfrak{S}_n\subset \mathsf{PC}(G)$ such that $\Phi_G \circ f_G(F)=F$ for every $F\in \mathsf{ISF}(G)$. In particular $\Phi_G$ is surjective, $f_G$ is injective, and $\mathsf{ISF}(G)=\{\Phi_G(\sigma)\mid \sigma\in \mathfrak{S}_n\}=\{\Phi_G(\sigma)\mid \sigma\in f_G(\mathsf{ISF}(G))\}$.
\end{theorem}
\begin{proof}
It is easy to check that Algorithm~\ref{alg:PhiG_inverse} is well defined and that it always terminates. Indeed the first \textbf{for} loop is just a sequence of iterations of Algorithm~\ref{alg:PhiG_inverse_tree} on the trees $T_1,T_2,\dots$ of $F\in \mathsf{ISF}(G)$, in such a way that the tree $T_i$ gets the colors $\{n-\sum_{j=1}^{i}|V(T_j)|+1,n-\sum_{j=1}^{i}|V(T_j)|+2,\dots,n-\sum_{j=1}^{i-1}|V(T_j)|\}$. Then the second \textbf{for} loop takes the colorings computed before and put it together to get a coloring of the forest $F$, getting obviously a permutation $\sigma\in \mathfrak{S}_n$.

To see why $\Phi_G \circ f_G(F)=F$, notice that, because of how we attributed the colors to the trees,  the first call of $\mathrm{getW}$ (Algorithm~\ref{alg:getW}) in Algorithm~\ref{alg:main} will necessarily pick the vertices of $T_1$ (the colors of the vertices of the other trees are all strictly smaller than all these colors), hence it will send the colorings of those vertices back to the tree $T_1$ (by Proposition~\ref{prop:PhiG_inverse_tree}); then the second call of $\mathrm{getW}$ will necessarily pick the vertices of $T_2$ (the colors of the vertices of the other remaining trees are all strictly smaller than all these colors), hence it will send the colorings of those vertices back to the tree $T_2$ (by Proposition~\ref{prop:PhiG_inverse_tree}); and so on.
\end{proof}

Our function $f_G$ has also the following property.
\begin{proposition}
	For every graph $G=([n],E)$ the function $f_G:\mathsf{ISF}(G)\to \mathfrak{S}_n$ defined by Algorithm~\ref{alg:PhiG_inverse} is such that $\mathsf{Des}_G(f_G(F)^{-1})=[n-1]$ for every $F\in \mathsf{ISF}(G)$.
\end{proposition}
\begin{proof}
	Suppose by contradiction that there exists $G=([n],E)$, $F\in \mathsf{ISF}(G)$ and $i\in [n-1]\setminus \mathsf{Des}_G(f_G(F)^{-1})$. For simplicity, we set $\sigma:=f_G(F)$, so that we must have $\sigma^{-1}(i)<\sigma^{-1}(i+1)$ and $\{\sigma^{-1}(i),\sigma^{-1}(i+1)\}\notin E$. 
	
	We use the fact proved in Theorem~\ref{thm:PhiG_surjective} that $\Phi_G(f_G(F))=F$. We distinguish two cases.
	
	Case 1: both $\sigma^{-1}(i)$ and $\sigma^{-1}(i+1)$ are in the same tree $T$ of $F$. Then Algorithm~\ref{alg:PhiG_inverse_tree} tells us how $\Phi_G(f_G(F))=F$ could have happened: in the $\sigma^{-1}(i+1)$-th iteration of the first \textbf{for} loop there was a color $a\leq i$ in position $\sigma^{-1}(i)$, and in that iteration we added $a+1$ in position $\sigma^{-1}(i+1)$; later on we might have updated those colors by adding $1$ to both of them simultaneously. But the problem is that to add $a+1$ the color $a$ should have been in the set $cL$ (see Algorithm~\ref{alg:PhiG_inverse_tree}), hence the vertex $\sigma^{-1}(i)$ should have been in $L$, but this is impossible since $\{\sigma^{-1}(i),\sigma^{-1}(i+1)\}\notin E$.
	
	Case 2: $\sigma^{-1}(i)$ and $\sigma^{-1}(i+1)$ occur in distinct trees of $F$, say $T$ and $T'$ respectively. Since $\Phi_G(f_G(F))=F$ and Algorithm~\ref{alg:PhiG_inverse} gives higher colors to trees that are to the left in $F$, we must have that $T'$ is to the left of $T$, i.e.\ $root(T')<root(T)$. Moreover $i+1$ must be the smallest color occurring in $T'$, since the colors of every tree of $F$ are consecutive (cf.\ Proposition~\ref{prop:PhiG_inverse_tree}), and the smallest color always gets assigned to the root $root(T')$ of the tree (cf.\ Algorithm~\ref{alg:getW}). So $\sigma^{-1}(i+1)=root(T')<root(T)$, and $root(T)$ is smaller than any other vertex of $T$, hence it is also smaller than $\sigma^{-1}(i)$, giving a contradiction.
\end{proof}

The function $\Phi_G$ has also the following important property.

\begin{proposition} \label{prop:Phi_wt_coinv}
Given a graph $G=([n],E)$, the function $\Phi_G:\mathsf{PC}(G)\to \mathsf{ISF}(G)$ defined by Algorithm~\ref{alg:main} is such that $\mathsf{wt}_G(\Phi_G(\kappa))\geq \mathsf{coinv}_G(\kappa)$ for every $\kappa\in \mathsf{PC}(G)$.
\end{proposition}
\begin{proof}
Notice that in the \textbf{for} loop we added $W_i$ with an edge $(u,W_i)$ in such a way that the weight $\mathsf{wt}_G((u,W_i))$ equals the number of $G$-coinversions involving $W_i$ and the elements in $W$ smaller than $W_i$. In this way the number of $G$-coinversions of $\kappa$ involving two vertices of the same tree $T_j$ equals $\mathsf{wt}_G(T_j)$. The other $G$-coinversions of $\kappa$ are automatically $G$-inversions of $\Phi_G(\kappa)$: if $u<w$, $(u,w)\in E$, $\kappa(u)<\kappa(w)$, $u\in T_j$ and $w\in T_s$ with $j\neq s$, then necessarily $j>s$, in other words $w$ gets added before $u$ in Algorithm~\ref{alg:main} (if we add $u$ first, then Algorithm~\ref{alg:getW} would put $w$ in the same tree as $u$). This shows that $\mathsf{wt}_G(\Phi_G(\kappa))\geq \mathsf{coinv}_G(\kappa)$.
\end{proof}
\begin{remark} \label{rmk:missing_equality}
In the above proof of Proposition~\ref{prop:Phi_wt_coinv}, the only thing that is missing to get the equality $\mathsf{wt}_G(\Phi_G(\kappa))= \mathsf{coinv}_G(\kappa)$ is the fact that every $G$-inversion of $\Phi_G(\kappa)$ corresponds to a $G$-coinversion of $\kappa$. In fact this is not always the case: for example for the graph $G=([3],\{(1,3),(2,3)\}$ (which is not in $\mathcal{IG}_3$) we have $\Phi_G(132)=((\{1,3\},\{(1,3)\}),(\{2\},\varnothing))$ and $\mathsf{CoInv}_G(132)=\{(1,3)\}$, so that $\mathsf{wt}_G(\Phi_G(132))=2>1=\mathsf{coinv}_G(132)$. 
\end{remark}

\begin{remark} \label{rmk:CoInvG_determines_Phi}
An important observation is that $\Phi_G(\kappa)$ only depends on $\mathsf{CoInv}_G(\kappa)$: in other words, if $\kappa,\kappa'\in \mathsf{PC}(G)$ are such that $\mathsf{CoInv}_G(\kappa)=\mathsf{CoInv}_G(\kappa')$, then $\Phi_G(\kappa)=\Phi_G(\kappa')$. Indeed, in Algorithm~\ref{alg:main} $\kappa$ only enters in two places: in the call of $\mathrm{getW}$, and it is immediate to see that Algorithm~\ref{alg:getW} is determined by the $G$-coinversions of $\kappa$ (cf.\ the condition of the \textbf{if}), and in the computation of $r$ in the \textbf{for} loop, and again it is clear that it only depends on the $G$-coinversions of $\kappa$.
\end{remark}

In the special case of $G=([n],E)\in \mathcal{IG}_n$ an interval graph, the statement in Remark~\ref{rmk:CoInvG_determines_Phi} has a converse. In fact we have even more.
\begin{proposition} \label{prop:Phi_Coinv_inverval}
Given an interval graph $G=([n],E)\in \mathcal{IG}_n$, the function $\Phi_G:\mathsf{PC}(G)\to \mathsf{ISF}(G)$ defined by Algorithm~\ref{alg:main} is such that for every $\kappa,\kappa'\in \mathsf{PC}(G)$, $\Phi_G(\kappa)=\Phi_G(\kappa')$ if and only if $\mathsf{CoInv}_G(\kappa)=\mathsf{CoInv}_G(\kappa')$. Moreover $\mathsf{wt}_G(\Phi_G(\kappa))= \mathsf{coinv}_G(\kappa)$ for every $\kappa\in \mathsf{PC}(G)$.
\end{proposition}
\begin{proof}
For the first statement, because of Remark~\ref{rmk:CoInvG_determines_Phi}, it is enough to show that we can reconstruct $\mathsf{CoInv}_G(\kappa)$ from the knowledge of $G$ and $\Phi_G(\kappa)$.

(i) Inside each tree of $\Phi_G(\kappa)$ we can reconstruct the $G$-coinversions inductively. Indeed, when we add a vertex in Algorithm~\ref{alg:main}, we are in the $i$-th instance of a \textbf{for} loop, and we are adding $W_i$ to the tree $T$; the vertices in $T$ that are connected to $W_i$ are the vertices in $L$ (they are all smaller than $W_i$), and since $G$ is an interval graph they form a clique (cf.\ Remark~\ref{rmk:CoInvG_determines_Phi}); from the vertex which $W_i$ is attached so we can reconstruct $r$, so that we can determine how many $u\in L$ are such that $\kappa(u)<\kappa(W_i)$; since by induction we know the order of $\{\kappa(u)\mid u\in L\}$, we determine in this way the order of $\{\kappa(u)\mid u\in L\cup\{W_i\}\}$.

(ii) Between trees, assume that $u$ is in a tree $T$ strictly to the left of the tree $T'$ containing $v$, $u>v$ and $(v,u)\in E$. We already observed in the proof of Proposition~\ref{prop:PhiG_welldefined} that the vertices of $T$ are determined by the call of the function $\mathrm{getW}$, hence (see Algorithm~\ref{alg:getW}), since $u\in T$, there exists a sequence $root(T)=u_1<u_2<\cdots<u_r=u$ of vertices of $T$ such that $(u_i,u_{i+1})\in E$ and $\kappa(u_i)<\kappa(u_{i+1})$ for every $i\in [r-1]$. Since $root(T)<root(T')$, $root(T')$ is smaller than every other vertex of $T'$ and $v$ is one of those vertices, there exists a $j\in [r-1]$ such that $u_j<v<u_{j+1}$. Since $G$ is an interval graph, $(v,u=u_r)\in E$ and $v<u_{j+1}<u$, we have $(v,u_{j+1})\in E$. Again, since $G$ is an interval graph, $(u_j,u_{j+1})\in E$ and $u_j<v<u_{j+1}$, we have $(u_j,v)\in E$. So if $\kappa(u)<\kappa(v)$, we would have \[\kappa(u_j)<\kappa(u_{j+1})<\cdots <\kappa(u_r)=\kappa(u)<\kappa(v),\]
and Algorithm~\ref{alg:main} would have placed $v$ in $T$, which is a contradiction; hence we must have $\kappa(u)>\kappa(v)$.

(iii) Finally, if $u$ is in a tree $T$ strictly to the left of the tree $T'$ containing $v$, $u<v$ and $(u,v)\in E$, then necessarily $\kappa(u)>\kappa(v)$, otherwise the algorithm would have placed $v$ in $T$, which is a contradiction. 

This shows that we can reconstruct all the $G$-coinversions of $\kappa$ from $G$ and $\Phi_G(\kappa)$, as we wanted. 

For the last statement, point (ii) above shows that a $G$-inversion of $\Phi_G(\kappa)$ must necessarily correspond to a $G$-coinversion of $\kappa$, and this was the only thing that was missing in the proof of Proposition~\ref{prop:Phi_wt_coinv} to get the actual equality $\mathsf{wt}_G(\Phi_G(\kappa))= \mathsf{coinv}_G(\kappa)$: see Remark~\ref{rmk:missing_equality}.
\end{proof}
\begin{remark}
	In Proposition~\ref{prop:Phi_Coinv_inverval} we cannot drop the hypothesis $G\in \mathcal{IG}_n$ even for the first statement. 
	For example the graph $G=([3],\{(1,3),(2,3)\}$ is not in $\mathcal{IG}_3$, and now $\Phi_G(123)=\Phi_G(132)=((\{1,3\},\{(1,3)\}),(\{2\},\varnothing))$ but $\mathsf{CoInv}_G(123)=\{(1,3),(2,3)\}\neq\{(1,3)\}=\mathsf{CoInv}_G(132)$.
\end{remark}

\begin{remark} \label{rmk:CoInvG_of_F}
When $G=([n],E)\in\mathcal{IG}_n$ is an interval graph, using Theorem~\ref{thm:PhiG_surjective} we can define for any $F\in\mathsf{ISF}(G)$ a set of \emph{$G$-coinversions} of $F$ by setting $\mathsf{CoInv}_G(F):=\mathsf{CoInv}_G(f_G(F))$, where $f_G$ is the function defined by Algorithm~\ref{alg:PhiG_inverse}. Then Proposition~\ref{prop:Phi_Coinv_inverval} shows that for every $\kappa\in\mathsf{PC}(G)$, $\phi(\kappa)=F$ if and only if $\mathsf{CoInv}_G(\kappa)=\mathsf{CoInv}_G(F)$.
\end{remark}

Here is an application of our results.

\begin{corollary}
For $G=K_n$ the complete graph, the restriction $\Phi_{K_n}|_{\mathfrak{S}_n}:\mathfrak{S}_n\to \mathsf{ISF}(K_n)$ is a bijection, and
\[[n]_q!=\sum_{F\in \mathsf{ISF}(K_n)}q^{\mathsf{wt}_{K_n}(F)}.\]
\end{corollary}
\begin{proof}
The first statement is clear: by Theorem~\ref{thm:PhiG_surjective} the restriction is always surjective, and when $G=K_n$ the set $\mathsf{CoInv}_G(\sigma)$ is simply the set of coinversions of $\sigma$, which is well known to determine $\sigma$ uniquely (essentially via the \emph{Lehmer code}).

The formula follows from the property that $\mathsf{wt}_{K_n}(\Phi_{K_n}(\sigma))=\mathsf{coinv}_{K_n}(\sigma)$ (cf.\ Proposition~\ref{prop:Phi_Coinv_inverval}), and the well-known fact that the number of (co)inversions is a Mahonian statistic.
\end{proof}

We are now ready to define quasisymmetric functions associated to increasing spanning forests of interval graphs.

\begin{definition}
Given an interval graph $G=([n],E)\in \mathcal{IG}_n$, and given $F\in \mathsf{ISF}(G)$, we define the formal power series
\[\mathcal{Q}_F^{(G)}=\mathcal{Q}_F^{(G)}[X]:=\mathop{\sum_{\kappa\in \mathsf{PC}(G)}}_{\Phi_G(\kappa)=F}x_\kappa.\]
\end{definition}
We give immediately the following fundamental formula.
\begin{theorem} \label{thm:forest_fundamental}
Given an interval graph $G=([n],E)\in \mathcal{IG}_n$, and given $F\in \mathsf{ISF}(G)$, we have
\begin{equation}
	\mathcal{Q}_F^{(G)}=\mathop{\sum_{\sigma\in \mathfrak{S}_n}}_{\mathsf{CoInv}_G(\sigma)=\mathsf{CoInv}_G(F)}L_{n,\mathsf{Des}_G(\sigma^{-1})}.
\end{equation}
\end{theorem}
\begin{proof}
Combining Proposition~\ref{prop:Phi_Coinv_inverval} and Theorem~\ref{thm:PhiG_surjective}, we have
\[\mathcal{Q}_F^{(G)}=\mathop{\sum_{\kappa\in \mathsf{PC}(G)}}_{\mathsf{CoInv}_G(\kappa)=\mathsf{CoInv}_G(F)} x_\kappa,\]
and now the formula follows from Proposition~\ref{prop:inv_coinv_formulae} (at $q=1$).
\end{proof}

For example, for $G=([8],E)$ the interval graph in Figure~\ref{fig:interval_graph} and $F\in \mathsf{ISF}(G)$ the forest in Figure~\ref{fig:isf_example}, we get
\begin{align*}
\mathcal{Q}_F^{(G)} & =L_{(1^8)} + L_{(1^6, 2)} + 2 L_{(1^5, 2, 1)} + 3 L_{(1^4, 2, 1^2)} + 2 L_{(1^4, 2^2)}+ L_{(1^4, 3, 1)} + 3 L_{(1^3, 2, 1^3)}\\
&  + 3 L_{(1^3, 2, 1, 2)} + 2 L_{(1^3, 2^2, 1)} + L_{(1^3, 3, 1^2)}+ L_{(1^3, 3, 2)} + 3 L_{(1^2, 2, 1^4)} + 3 L_{(1^2, 2, 1^2, 2)}\\
 &   + 2 L_{(1^2, 2, 1, 2, 1)} + 2 L_{(1^2, 2^2, 1^2)} + 2 L_{(1^2, 2^3)}+ L_{(1^2, 3, 1^3)} + L_{(1^2, 3, 1, 2)} + 2 L_{(1, 2, 1^5)} \\
 &  + 2 L_{(1, 2, 1^3, 2)} + 2 L_{(1, 2, 1^2, 2, 1)} + 2 L_{(1, 2, 1, 2, 1^2)}+ 2 L_{(1, 2, 1, 2^2)} + 2 L_{(1, 2^2, 1^3)} + 2 L_{(1, 2^2, 1, 2)}\\
 &   + L_{(1, 3, 1^4)} + L_{(1, 3, 1^2, 2)}.
\end{align*}

For a smaller example, consider the interval graph $G=([3],\{(1,2),(1,3)\})\in \mathcal{IG}_3$ and the forest $F=([3],\{(1,2),(1,3)\})\in \mathsf{ISF}(G)$. Then $\mathsf{CoInv}_G(F)=\{(1,2),(1,3)\}$, and $\{\sigma\in \mathfrak{S}_3\mid \mathsf{CoInv}_G(\sigma)=\mathsf{CoInv}_G(F)\}=\{\sigma=123,\tau=132\}$, $\mathsf{Des}_G(\sigma^{-1})=\{1\}$, $\mathsf{Des}_G(\tau^{-1})=\{1,2\}$, hence
\[\mathcal{Q}_F^{(G)} = L_{(1,2)}+L_{(1^3)}.\]
\smallskip

Notice that Theorem~\ref{thm:forest_fundamental} shows in particular that our $\mathcal{Q}_F^{(G)}$ are quasisymmetric functions. 

In fact it is not hard to write the expansion of $\mathcal{Q}_F^{(G)}$ in the monomial quasisymmetric functions. Given $G=([n],E)$ a graph, call $\mathsf{PPC}(G)\subset \mathsf{PC}(G)$ the set of \emph{packed proper colorings} of $G$, i.e.\ proper colorings $\kappa\in \mathsf{PC}(G)$ such that $\{\kappa(i)\mid i\in [n]\}=[r]$ for some $1\leq r\leq n$.  Given $\kappa\in \mathsf{PPC}(G)$, call $\mathsf{ev}(\kappa)$ the \emph{exponents vector}, i.e.\ the composition of $n$ whose $i$-th part is the number of $j\in [n]$ such that $\kappa(j)=i$, for every $i=1,2,\dots,\max(\{\kappa(i)\mid i\in [n]\})$.

The next proposition follows immediately from the definition of $\mathcal{Q}_F^{(G)}$ and the fact that $\mathcal{Q}_F^{(G)}\in \mathrm{QSym}$.
\begin{proposition} \label{prop:Q_in_monomials}
Given an interval graph $G=([n],E)\in \mathcal{IG}_n$, and given $F\in \mathsf{ISF}(G)$, we have
\begin{equation}
	\mathcal{Q}_F^{(G)}=\mathop{\sum_{\kappa\in \mathsf{PPC}(G)}}_{\Phi_G(\kappa)=F}M_{\mathsf{ev}(\kappa)}.
\end{equation}
\end{proposition}
For example, consider the interval graph $G=([3],\{(1,2),(1,3)\})\in \mathcal{IG}_3$ and the increasing spanning forest $F=([3],\{(1,2),(1,3)\})\in \mathsf{ISF}(G)$. Then $\mathsf{CoInv}_G(F)=\{(1,2),(1,3)\}$ and \[\{\kappa\in\mathsf{PPC}(G)\mid \Phi_G(\kappa)=F\}=\{\kappa\in\mathsf{PPC}(G)\mid \mathsf{CoInv}_G(\kappa)=\mathsf{CoInv}_G(F)\}=\{123,132,122\},\]
giving
\[\mathcal{Q}_F^{(G)} = M_{(1,2)}+2M_{(1^3)},\]
which agrees with our previous computation.

\section{Interval orders, chromatic quasisymmetric functions and LLT} \label{sec:chromatic_LLT}

Given any simple graph $G=([n],E)$, Shareshian and Wachs defined in \cite{Shareshian_Wachs_Advances} its \emph{chromatic quasisymmetric function} as
\[\chi_G[X;q]:=\sum_{\kappa \in \mathsf{PC}(G)}q^{\mathsf{coinv}_G(\kappa)}x_\kappa.\]
\begin{remark}
Notice that in the literature, following the original work \cite{Shareshian_Wachs_Advances}, what we call $\mathsf{coinv}_G$ ($\mathsf{inv}_G$ respectively) is called a $\mathsf{asc}_G$ ($\mathsf{des}_G$ respectively). We prefer to avoid the reminescence of the  words ``ascents'' and ``descents'', as they are commonly used (also in the present work) with different meanings.
\end{remark}

The following theorem is a direct consequence of Proposition~\ref{prop:Phi_Coinv_inverval} and Theorem~\ref{thm:PhiG_surjective}.
\begin{theorem} \label{thm:chrom_forests_formula}
Given an interval graph $G=([n],E)\in \mathcal{IG}_n$, we have
\begin{equation}
\chi_G[X;q]=\sum_{F\in \mathsf{ISF}(G)}q^{\mathsf{wt}_G(F)}\mathcal{Q}_F^{(G)}.
\end{equation}
\end{theorem} 
\begin{example}\label{ex:chrom_csf}
For $G=([3],\{(1,2),(1,3)\})$, the increasing spanning forests of $G$ are 
\begin{align*}
	& F_1=(([3],\{(1,2),(1,3)\})),\quad  F_2=((\{1,3\},\{(1,3)\}),(\{2\},\varnothing)),\\  &F_3=((\{1,2\},\{(1,2)\}),(\{3\},\varnothing)),\quad  F_4=((\{1\},\varnothing),(\{2\},\varnothing),(\{3\},\varnothing)),
\end{align*} 
and we compute
\[\mathsf{wt}_G(F_1)=2,\quad \mathsf{wt}_G(F_2)=\mathsf{wt}_G(F_3)=1,\quad  \mathsf{wt}_G(F_4)=0,\]  \[\mathcal{Q}_{F_1}^{(G)}=L_{(1,2)}+L_{(1^3)}, \quad  \mathcal{Q}_{F_2}^{(G)}=\mathcal{Q}_{F_3}^{(G)}=L_{(1^3)},\quad \mathcal{Q}_{F_4}^{(G)}=L_{(2,1)}+L_{(1^3)},\] hence finally \[\chi_G[X;q]=L_{(2,1)}+q^2L_{(1,2)}+(1+2q+q^2)L_{(1^3)}.\]
\end{example}

The following corollary is a reformulation of \cite{Shareshian_Wachs_Advances}*{Theorem~3.1} in our cases, and it follows immediately from Theorems~\ref{thm:chrom_forests_formula} and \ref{thm:forest_fundamental}. 
\begin{corollary} \label{cor:chrom_fundamentals}
Given an interval graph $G=([n],E)\in \mathcal{IG}_n$, we have
\begin{equation} \label{eq:cor_chiG_coinv}
	\chi_G[X;q]=\sum_{\sigma\in \mathfrak{S}_n}q^{\mathsf{coinv}_G(\sigma)}L_{n,\mathsf{Des}_G(\sigma^{-1})}.
\end{equation}
\end{corollary}
\begin{remark} \label{rmk:SW_notation}
To see why Corollary~\ref{cor:chrom_fundamentals} is a reformulation of \cite{Shareshian_Wachs_Advances}*{Theorem~3.1}, for any $\sigma\in \mathfrak{S}_n$ set $\overline{\sigma}=\sigma(n)\sigma(n-1)\cdots \sigma(2)\sigma(1)$, i.e.\ $\overline{\sigma}(i):=\sigma(n+1-i)$ for every $i\in [n]$. If for every $\sigma\in \mathfrak{S}_n$ we set\footnote{In \cite{Shareshian_Wachs_Advances} they call ``$\mathrm{inv}_G$'' what we call here ``$\widetilde{\mathsf{inv}}_G$'' and ``$\mathrm{DES}_P$'' essentially what we call here ``$\widetilde{\mathsf{Des}}_G$'' ($P$ refers to the poset ``behind'' $G$).}
\[\widetilde{\mathsf{Des}}_G(\sigma):=\{i\in [n-1]\mid \sigma(i)>\sigma(i+1)\text{ and }\{\sigma(i),\sigma(i+1)\}\notin E\}\]
and
\[\widetilde{\mathsf{inv}}_G(\sigma):=\{\{\sigma(i),\sigma(j)\}\in E\mid i<j \text{ and } \sigma(i)>\sigma(j)\},\]
then \cite{Shareshian_Wachs_Advances}*{Theorem~3.1} can be restated as
\[\psi \chi_G[X;q] = \sum_{\sigma\in \mathfrak{S}_n}q^{\widetilde{\mathsf{inv}}_G(\sigma)}L_{n,n-\widetilde{\mathsf{Des}}_G(\sigma)}\]
recalling that for every $S\subseteq [n-1]$, $n-S:=\{n-i\mid i\in S\}$, and $\psi$ is the involution sending $L_{n,S}$ to $L_{n,[n-1]\setminus S}$.

It is now easy to check that 
\begin{equation}\label{eq:Destilde}
[n-1]\setminus (n-\widetilde{\mathsf{Des}}_G(\sigma))=\mathsf{Des}_G(\overline{\sigma})
\end{equation}
and  
\begin{equation}\label{eq:invtilde}
\widetilde{\mathsf{inv}}_G(\sigma)=\mathsf{inv}_G(\sigma^{-1})=\mathsf{coinv}_G(\overline{\sigma}^{-1}).
\end{equation}
\end{remark}

Using Propositions~\ref{prop:Phi_Coinv_inverval} and \ref{prop:Q_in_monomials}, from Theorem~\ref{thm:chrom_forests_formula} we also get the following formula.
\begin{corollary}
Given an interval graph $G=([n],E)\in \mathcal{IG}_n$, we have
\[ \chi_G[X;q]=\sum_{\kappa\in \mathsf{PPC}(G)}q^{\mathsf{coinv}_G(\kappa)}M_{\mathsf{ev}(\kappa)}. \]
\end{corollary}
\smallskip

Given any simple graph $G=([n],E)\in \mathcal{IG}_n$, we define its \emph{LLT quasisymmetric function} as
\[\mathrm{LLT}_G[X;q]:=\sum_{\kappa \in \mathsf{C}(G)}q^{\mathsf{inv}_G(\kappa)}x_\kappa.\]

The following formula is an immediate consequence of Proposition~\ref{prop:invG_LLT}.
\begin{theorem}\label{thm:LLT_fundamental}
Given any simple graph $G=([n],E)$, we have
\[\mathrm{LLT}_G[X;q]=\sum_{\sigma\in \mathfrak{S}_n}q^{\mathsf{inv}_G(\sigma)}L_{n,\mathsf{Des}(\sigma^{-1})}.\]
\end{theorem}
For example, if $G=([3],\{(1,2),(1,3)\})\in \mathcal{IG}_3$, then
\[\mathrm{LLT}_G[X;q]=q^2L_{(1^3)}+(q+q^2)L_{(1,2)}+(1+q)L_{(2,1)}+L_{(3)}.\]

\smallskip 

The name LLT of these quasisymmetric functions comes from the following well-known facts: when $G$ is a Dyck graph, $\mathrm{LLT}_G[X;q]$ is a symmetric function, and in fact it is a so called \emph{unicellular LLT symmetric functions} (see e.g.\ \cite{Alexandersson-Panova}*{Section~3}). In this case the formula in Theorem~\ref{thm:LLT_fundamental} is well known (e.g.\ it can be deduced from \cite{Novelli_Thibon_NoncommLLT}*{Theorem~8.6})).

Notice that $\mathrm{LLT}_G[X;q]$ is typically not symmetric when $G$ is not a Dyck graph. In particular $\mathrm{LLT}_G[X;q]$ is typically not symmetric when $G\in \mathcal{IG}_n\setminus \mathcal{DG}_n$ (e.g.\ Example~\ref{ex:chrom_csf} is manifestly not symmetric). Still, we argue (cf.\ Theorem~\ref{thm:main_theorem}) that the LLT quasisymmetric functions of interval graphs are an interesting \emph{quasisymmetric} extension of the family of \emph{unicellular LLT} (for a noncommutative analogue see instead \cite{Novelli_Thibon_NoncommLLT}).

\section{A fundamental formula} \label{sec:key_formula}

Let $G=([n],E)\in \mathcal{IG}_n$ be an interval graph.

Given a composition $\beta\vDash n$, let $R(\beta)$ be the set of words of length $|\beta|=n$ in the alphabet $[\ell(\beta)]$ in which $i$ occurs $\beta_i$ times for every $i\in [\ell(\beta)]$. In other words, $R(\beta)$ is the set of all the rearrangements of the word $1^{\beta_1}2^{\beta_2}\cdots \ell(\beta)^{\beta_{\ell(\beta)}}$. Of course we can identify $R(\beta)$ with a subset of colorings $\mathsf{C}(G)$, so that for every $w\in R(\beta)$ 
\[ \inv_G(w)=\#\{\{i,j\}\in E\mid i<j\text{ and }w_i>w_j\}. \]

Recall all the definitions from Section~\ref{sec:qsym}, and the definitions of $\widetilde{\mathsf{inv}}_G$ and $\widetilde{\mathsf{Des}}_G$ from Remark~\ref{rmk:SW_notation}: for every $\sigma\in \mathfrak{S}_n$
\[\widetilde{\mathsf{Des}}_G(\sigma):=\{i\in [n-1]\mid \sigma(i)>\sigma(i+1)\text{ and }\{\sigma(i),\sigma(i+1)\}\notin E\}\]
and
\[\widetilde{\mathsf{inv}}_G(\sigma):=\{\{\sigma(i),\sigma(j)\}\in E\mid i<j \text{ and } \sigma(i)>\sigma(j)\}.\]

Finally, set
\begin{equation} \label{eq:alphaG_definition}
\alpha_G(\sigma):= \mathsf{comp}(n-[n-1]\setminus \mathsf{Des}_G(\overline{\sigma}))=\mathsf{comp}\left(\widetilde{\mathsf{Des}}_G(\sigma)\right),
\end{equation}
where the last equality follows from \eqref{eq:Destilde}.

For example, if $G$ is the graph in Figure~\ref{fig:interval_graph} and $\sigma=31852647\in \mathfrak{S}_8$, then $\mathsf{Des}_G(\overline{\sigma})=\{1,3,4,6\}$, so that $\widetilde{\mathsf{Des}}_G(\sigma)=8-[8-1]\setminus \mathsf{Des}_G(\overline{\sigma})=\{1,3,6\}$, and $\alpha_G(\sigma)=(1,2,3,2)$.

We are ready to state the main result of this section.
\begin{theorem} \label{thm:MainFormula}
	For every interval graph $G=([n],E)\in \mathcal{IG}_n$ and every $\beta\vDash n$
	\begin{equation} \label{eq:MainFormula} \sum_{\alpha\vDash n}q^{\eta(\gamma(\alpha,\beta))}\mathop{\sum_{\sigma\in\mathfrak{S}_n}}_{\alpha_G(\sigma)=\alpha}q^{\widetilde{\inv}_G(\sigma)} =[\beta_1]_q! [\beta_2]_q!\cdots [\beta_{\ell(\beta)}]_q!\sum_{w\in R(\beta)}q^{\inv_G(w)}.
	\end{equation}
\end{theorem}
To prove this theorem, we first prove the special case $\beta=(n)$, and then we see how to reduce the general case to this one.

For $\beta=(n)$, we have $\gamma(\alpha,\beta)=\alpha$, so that \[\eta(\gamma(\alpha,\beta))=\eta(\alpha)=\sum_{i=1}^{\ell(\alpha)-1}\sum_{j=1}^i\alpha_j,\]
hence if for every $\sigma\in \mathfrak{S}_n$ we set
\[ \widetilde{\maj}_{G^c}(\sigma):=\eta(\alpha_G(\sigma))=\sum_{i\in \widetilde{\mathsf{Des}}_G(\sigma)}i ,\]
then \eqref{eq:MainFormula} reduces to the following identity.
\begin{theorem} \label{thm:MainFormula_beta_(n)}
For every interval graph $G=([n],E)\in \mathcal{IG}_n$ we have
\[\sum_{\sigma\in\mathfrak{S}_n}q^{\widetilde{\maj}_{G^c}(\sigma)+\widetilde{\inv}_G(\sigma)}=[n]_q!\ .\]
\end{theorem}

It turns out that this identity is a reformulation of a theorem in \cite{Kasraoui_maj-inv} (cf. \cite{Kasraoui_maj-inv}*{Corollary~1.11}). Since it is not immediate to see how the theorem in \cite{Kasraoui_maj-inv} translates into our statement, we will explain in the next section how the bijective proof goes. In this section we show how this theorem implies the general formula.

\begin{proof}[Proof of Theorem~\ref{thm:MainFormula}]
	Let $\beta\vDash n$, $\beta=(\beta_1,\beta_2,\dots,\beta_r)$ so that $r=\ell(\beta)$. We will use the notation $\llbracket a,b\rrbracket:=\{a,a+1,\dots,b-1,b\}$. For a fixed $\sigma\in\mathfrak{S}_n$ and for every $i\in [r]$, let 
\[S_i^{\sigma,\beta}=S_i:=\left\{\sigma(j)\mid j\in \left\llbracket 1+\sum_{k=1}^{i-1}\beta_k,\sum_{k=1}^{i}\beta_k\right\rrbracket\right\},\]
and let $\sigma_\beta$ be the element in $R(\beta)$ obtained by replacing in the word $123\cdots n$ every element in $S_i$ with an $i$, for every $i\in [r]$. 

For example, if $\beta=(1,2,1,3)\vDash 7$, so that $r=\ell(\beta)=4$, and $\sigma=4163275\in \mathfrak{S}_7$, then $S_1=\{4\}$, $S_2=\{1,6\}$, $S_3=\{3\}$, $S_4=\{2,5,7\}$, and $\sigma_\beta=2431424\in R(\beta)$.

Now call $\widetilde{\inv}_G(\sigma,S_i)$ the contribution to $\widetilde{\inv}_G(\sigma)$ of the subword of $\sigma$ consisting of the letters in $S_i$. 

In the previous example, if $G=([7],\{(1,2),(1,3),(1,4),(2,3),(3,4),(4,5),(4,6),(5,6),(5,7)\})$, then
$\widetilde{\inv}_G(\sigma,S_1)=\widetilde{\inv}_G(\sigma,S_2)=\widetilde{\inv}_G(\sigma,S_3)=0$, while $\widetilde{\inv}_G(\sigma,S_4)=1$ (here $(5,7)$ is contributing).

Then the observation is that
\[ \widetilde{\inv}_G(\sigma)= \inv_G(\sigma_\beta)+\sum_{i=1}^r\widetilde{\inv}_G(\sigma,S_i^{\sigma,\beta}). \]
Indeed the contributions to $\widetilde{\inv}_G(\sigma)$ not counted in the summation in the right hand side correspond precisely to the pairs $\{\sigma(i),\sigma(j)\}\in E$ such that $1\leq i<j\leq n$, $\sigma(i)>\sigma(j)$, $\sigma(i)\in S_k$ and $\sigma(j)\in S_h$, so that $k<h$. Now such a pair gives a $k$ in $\sigma_\beta$ occurring in position $\sigma(i)$ and an $h$ in position $\sigma(j)$. Hence all these pairs correspond precisely to the ones counted by $\inv_G(\sigma_\beta)$.

In our example, $\inv_G(\sigma_\beta)=4$ (here $(1,4),(2,3),(3,4)$ and $(5,6)$ are contributing), and indeed 
\[\widetilde{\inv}_G(\sigma)=5=4+0+0+0+1=\inv_G(\sigma_\beta)+\widetilde{\inv}_G(\sigma,S_1)+\widetilde{\inv}_G(\sigma,S_2)+\widetilde{\inv}_G(\sigma,S_3)+\widetilde{\inv}_G(\sigma,S_4)\]

Therefore
\begin{align*}
& \sum_{\alpha\vDash n}q^{\eta(\gamma(\alpha,\beta))}\mathop{\sum_{\sigma\in\mathfrak{S}_n}}_{\alpha^G(\sigma)=\alpha}q^{\widetilde{\inv}_G(\sigma)}=\\
& =\sum_{\sigma\in\mathfrak{S}_n}q^{\eta(\gamma(\alpha_G(\sigma),\beta))+\widetilde{\inv}_G(\sigma)}\\
& =\sum_{\sigma\in\mathfrak{S}_n}q^{\inv_G(\sigma_\beta)+\sum_{i=1}^r\eta(\gamma^i(\alpha_G(\sigma),\beta))+\sum_{i=1}^r\widetilde{\inv}_G(\sigma,S_i^{\sigma,\beta})}\\
& =\sum_{w\in R(\beta)}q^{\inv_G(w)}\mathop{\sum_{\sigma\in\mathfrak{S}_n}}_{\sigma_\beta=w}q^{\sum_{i=1}^r\left[\eta(\gamma^i(\alpha_G(\sigma),\beta))+ \widetilde{\inv}_G(\sigma,S_i^{\sigma,\beta})\right]}.
\end{align*}
Now notice that in the internal sum, fixing $\sigma_\beta=w$ determines the $S_i^{\sigma,\beta}$, so that the permutations $\sigma\in \mathfrak{S}_n$ such that $\sigma_\beta=w$ are simply obtained by permuting in any way the elements of $S_1$ and placing them in the first $\beta_1=|S_1|$ positions, then doing the same with the elements of $S_2$ and putting those in the following $\beta_2$ positions, etc. Observe also that the statistic $\eta(\gamma^i(\alpha_G(\sigma),\beta))+ \widetilde{\inv}_G(\sigma,S_i^{\sigma,\beta})$ corresponds to the statistic $\eta(\alpha_{G(S_i)}(\tau))+\widetilde{\inv}_{G(S_i)}(\tau)$ on the permutations $\tau$ of elements of $S_i$ with respect to the graph $G(S_i)$ obtained from $G$ by taking the subgraph on the subset $S_i\subseteq [n]$ of vertices of $G$. Therefore the internal sum of the formula we just computed is indeed a product of formulas of the original type, but with $\beta=(n)$. Finally, it is easy to check that the $G(S_i)$ are still interval graphs (up to the obvious monotone relabelling of the vertices), hence we can apply Theorem~\ref{thm:MainFormula_beta_(n)} to finally get
\begin{align*}
& \sum_{w\in R(\beta)}q^{\widetilde{\inv}_G(w)}\mathop{\sum_{\sigma\in\mathfrak{S}_n}}_{\sigma_\beta=w}q^{\sum_{i=1}^r\left[\eta(\gamma^i(\alpha_G(\sigma),\beta))+ \widetilde{\inv}_G(\sigma,S_i^{\sigma,\beta})\right]}=\\
& =[\beta_1]_q! [\beta_2]_q!\cdots [\beta_{r}]_q!\sum_{w\in R(\beta)}q^{\widetilde{\inv}_G(w)}.
\end{align*}
\end{proof}

\section{A modified Foata bijection} \label{sec:Foata}

In this section we briefly describe how Kasraoui proved Theorem~\ref{thm:MainFormula_beta_(n)} in \cite{Kasraoui_maj-inv}, using our notation.

\smallskip

Let $G=([n],E)\in\mathcal{IG}_n$ be an interval graph. We describe a Foata-like bijection $\varphi_G$ 
such that $\inv(\varphi_G(\sigma))=\widetilde{\inv}_G(\sigma)+\widetilde{\maj}_{G^c}(\sigma)$ for every $\sigma\in \mathfrak{S}_n$.

The map $\varphi_G$ is described in terms of an auxiliary map $\gamma^G_x$ with $x\in [n]$, defined on the set of words $w$ in the alphabet $[n]$. First we define two sets \[R_x^w:=\{j\in [\ell(w)]\mid w_j>x\text{ and }\{x,w_j\}\notin E\}\quad \text{ and }\quad L_x^w:=\{j\in [\ell(w)]\mid w_j\leq x\text{ or }\{w_j,x\}\in E\},\]
where $\ell(w)$ is the \emph{length} of the word $w$, i.e.\ the number of its letters.

Notice that the sets $R_x^w$ and $L_x^w$ depend on $G$, even if it does not appear in the notation.
\begin{remark}\label{remark:order_P}
In our situation, for any $x\in [n]$ and for any word $w$, given $i\in R_x^w$ and $j\in L_x^w$, we must have $w_j<w_i$: $i\in R_x^w$ implies $x<w_i$ and $\{x,w_i\}\notin E$; if $w_j\geq w_i$, then $\{x,w_j\}\notin E$ since $G$ is an interval graph, but $j\in L_x^w$, hence necessarily $w_j<x$, which contradicts $x<w_i$. Therefore we must have $w_j< w_i$.
\end{remark}

We define $\gamma^G_x(\epsilon)=\epsilon$, where $\epsilon$ is the only word of lenght $0$, i.e.\ the empty word. 
Assume now that $\ell(w)>0$, where $w_{\ell(w)}$ is its rightmost element. In order to define $\gamma_x^G(w)$ we study two cases:
\begin{itemize}
	\item[a)]  $\ell(w)\in R_x^{w}$;
	\item[b)]  $\ell(w)\in L_x^{w}$.
\end{itemize} 

In case $a)$ we decompose $w$ as follows:
\begin{equation}\label{eq:decomposition}
	w=v^1x_{1}v^2x_{2}\ldots v^kx_k\qquad \text{ with }x_k:=w_{\ell(w)},
\end{equation}

where $x_1,\ldots,x_{k-1},x_k\in \{w_j\mid j\in R_x^{w}\}$ and $v^1,\ldots,v^k$ are (possibly empty) words in $\{w_j\mid j\in L_x^{w}\}$. We notice that the
decomposition in \eqref{eq:decomposition} is unique. We now define $\gamma_x^G(w)=x_1v^1x_2v^2\ldots x_kv^k$.

The situation in case $b)$ is the same except that we are changing the roles of the values of $R_x^w$ and $L_x^w$. We therefore decompose $w=v^1x_{1}v^2x_{2}\ldots v^kx_k$ with $x_k:=w_{\ell({w})}$, where now $x_1,\ldots,x_{k-1},x_k\in \{w_j\mid j\in L_x^{w}\}$ and $v^1,\ldots,v^k$ are (possibly empty) words in $\{w_j\mid j\in R_x^{w}\}$. We define again $\gamma_x^G(w)=x_1v^1x_2v^2\ldots x_kv^k$.

Now $\varphi_G$ is defined recursively as $\varphi_{G}(wx)=\gamma_x^G(\varphi_G(w))x$. So that, if $w=w_1w_2\cdots w_n$, then
\[ \varphi_G(w)=\gamma_{w_n}^G(\gamma_{w_{n-1}}^G(\cdots \gamma_{w_2}^G(\gamma_{w_1}^G(\epsilon)w_1)w_2)\cdots )w_{n-1})w_n. \]

We have the following result (the notation $\inv$ is explained in Remark~\ref{rmk:inv_for_Kn}).
\begin{proposition}[Kasraoui]\label{prop:foata_mod}
	If $G=([n],E)$ is an interval graph, then
	\[
	\varphi_G|_{\mathfrak{S}_n}:\mathfrak{S}_n\rightarrow \mathfrak{S}_n
	\]
	is a bijection such that for every $\sigma\in \mathfrak{S}_n$ we have  \[\inv(\varphi_G(\sigma))=\widetilde{\inv}_G(\sigma)+\widetilde{\maj}_{G^c}(\sigma).\]
\end{proposition}
Before recalling a proof of this proposition we illustrate how the map $\varphi_G$ works with an example.

\begin{example} \label{ex:Foata_interval}
	Let $n=6$ and $G=([6],\{(2,3),(2,4),(3,4)\})\in \mathcal{DG}_6\subset \mathcal{IG}_6$.
	We take $\sigma=512463\in \mathfrak{S}_6$ and notice that $\widetilde{\maj}_{G^c}(\sigma)=6$ (since $\widetilde{\mathsf{Des}}_G(\sigma)=\{1,5\}$) and $\widetilde{\inv}_G(\sigma)=1$ (the inversion is given by $(3,4)$).
	
	By definition
	\[\varphi_G(\sigma)=\gamma_3^G(\gamma_6^G(\gamma_4^G(\gamma_2^G(\gamma_1^G(\gamma_5^G(\epsilon)5)1)2)4)6)3. \]
	Hence we are going to build $\varphi_G(\sigma)$ in $6$ steps. The output of the $i$-th step is a word in $\{\sigma(1),\ldots,\sigma(i)\}$ (each value appearing exactly once) such that the rightmost value is $\sigma(i)$. 
	The first two steps are trivial:
	\[
	|5\rightarrow |5\rightarrow 5=\gamma_5^G(\epsilon)5\]
	\[
	|5|1\rightarrow |5|1 \rightarrow 51=\gamma_1^G(5)1.\]
	The third step can be more interesting. As we take in $\sigma(3)=2$ we compare it with $\sigma(2)=1$ and notice that $1\ngtr 2$. Therefore we draw a bar after every other element $z$ that is not greater than $2$ or such that $\{z,2\}\in E$, and we add always a bar at the end to the left (we will always do this later on without further notice). We now cyclically permute the elements in every block of numbers delimited by two consecutive bars, moving the rightmost element of the block in the leftmost position of the same block: in this particular case we are just permuting $1$ and $5$, i.e.
	\[|51|2\rightarrow |15|2\rightarrow 152=\gamma_2^G(51)2.\]
	
	Now to add $\sigma(4)=4$ we compare it with $\sigma(3)=2$ and find out $\{2,4\}\in E$. Again we draw a bar after every value $z$ that is not greater than $4$ or such that $\{z,4\}\in E$ and permute cyclically the elements in every block:
	$$
	|1|52|4\rightarrow |1|25|4\rightarrow 1254=\gamma_4^G(152)4.
	$$
	We proceed by comparing $4$ and $6$. Because $4\ngtr 6$ we draw a bar after every element $z$ that is not greater than $6$ or such that $\{z,6\}\in E$:
	$$
	|1|2|5|4|6 \rightarrow |1|2|5|4|6\rightarrow 12546=\gamma_6^G(1254)6.
	$$
	We conclude comparing $6$ and $3$. Because $6>3$ and $\{3,6\}\notin E$, we draw a bar after every element $z$ that is greater than $3$ and such that $\{3,z\}\notin E$. We obtain:
	\[
	|125|46|3 \rightarrow |512|64|3 \rightarrow 512643=\gamma_3^G(12546)3.
	\] 
	We finally obtained $\varphi_G(\sigma)=512643$. Notice that 
	\[\inv(\varphi_G(\sigma))=7=1+6=\widetilde{\inv}_G(\sigma)+\widetilde{\maj}_{G^c}(\sigma).\]
\end{example}
\begin{remark}
In the case of the graph $G=([n],\varnothing)$ with no edges, $\widetilde{\inv}_{G}\equiv 0$ while $\widetilde{\mathsf{maj}}_{G}(\sigma)=\mathsf{maj}(\sigma)=\sum_{i\in \mathsf{Des}(\sigma)}i$ is the usual \emph{major index} of a permutation $\sigma\in \mathfrak{S}_n$, and in this case $\varphi_{G}$ is the usual Foata bijection. After translating the notations, Proposition~\ref{prop:foata_mod} is a corollary of \cite{Kasraoui_maj-inv}*{Theorem~1.6}. In fact from \cite{Kasraoui_maj-inv}*{Corollary~1.11} we know that the interval graphs are precisely the graphs $G=([n],E)$ such that $\widetilde{\inv}_G+\widetilde{\maj}_{G^c}$ is a Mahonian statistic.
\end{remark}

\begin{proof}[Proof of Proposition \ref{prop:foata_mod}]
To simplify the notation, given a permutation $\sigma\in \mathfrak{S}_n$, we denote $\sigma(i)$ by $\sigma_i$ for every $i\in [n]$. Now notice that $\varphi_G(\sigma)$ is obtained recursively as follows:
	\[
	\varphi_G(\sigma)=\gamma^G_{\sigma_n}\left( \gamma^G_{\sigma_{n-1}}\left(\ldots\left(\gamma_{\sigma_2}^G(\gamma_{\sigma_1}^G(\epsilon)\sigma_1)\sigma_2\right)\ldots\right) \sigma_{n-1}\right)\sigma_n.
	\]
	We set $w^i=\gamma^G_{\sigma_{i}}\left(\ldots\left(\gamma_{\sigma_2}^G(\gamma_{\sigma_1}^G(\epsilon)\sigma_1)\sigma_2\right)\ldots\right) \sigma_{i}$ and notice that $\ell(w^i)=i$ and that the $i$ elements appearing in $w_i$ are precisely $\sigma_1,\ldots,\sigma_i$; furthermore, the rightmost element in $w^i$ is $\sigma_i$.
	
	Since the rightmost element in $w^i$ is $\sigma_i$, we can recover $w^{i-1}$ starting from $w^{i}$: we compare the leftmost letter of $w^i$ with $\sigma_i$, from which we can recover the position of the bars giving the blocks (cf.\ Example~\ref{ex:Foata_interval}), and hence we can perform inside each block the inverse cyclic permutation (bringing the leftmost element of each block in the rightmost place of the block), in order to recover $w^{i-1}$. Hence recursively we can recover $\sigma$ from $\varphi_G(\sigma)$. This shows the bijectivity of $\varphi_G|_{\mathfrak{S}_n}$.

	To show that $\inv(\varphi_G(\sigma))=\widetilde{\inv}_G(\sigma)+\widetilde{\maj}_{G^c}(\sigma)$ we prove the following:
	
	\noindent a) if $\sigma_{i}>\sigma_{i+1}$ and $\{\sigma_i,\sigma_{i+1}\}\notin E$, then \[\inv(w^{i+1})=\inv(\gamma^G_{\sigma_{i+1}}(w^i)\sigma_{i+1})=\inv(w^{i})+i+|\{j\in [i]\mid w_j^i>\sigma_{i+1}\text{ and } \{\sigma_{i+1},w^i_j\}\in E\}|;\]
	
	\noindent b) if $\sigma_{i}< \sigma_{i+1}$ or $\{\sigma_i,\sigma_{i+1}\}\in E$ then \[\inv(w^{i+1})=\inv(\gamma^G_{\sigma_{i+1}}(w^i)\sigma_{i+1})= \inv(w^i)+|\{j\in [i]\mid w_j^i>\sigma_{i+1}\text{ and }  \{\sigma_{i+1},w^i_j\}\in E\}|.\]
		
	We discuss the case $a)$ first. Starting from $w^i=w_1^iw_2^i\cdots w_i^i$, to obtain $\gamma^G_{\sigma_{i+1}}(w^i)$ we decompose $w^i=v^1x_1\cdots v^{k-1}x_{k-1}v^kx_k$ with $x_k:=\sigma_{i}$, where for every $j\in [k]$ we have $x_j> \sigma_{i+1}$ and $\{x_j,\sigma_{i+1}\}\notin E$, and $v^j$ is a word in elements $z$ that are not greater than $\sigma_{i+1}$ or $\{z,\sigma_{i+1}\}\in E$. By Remark~\ref{remark:order_P} we notice that for every $j\in [k]$, every letter in $v^j$ is smaller than $x_j$, therefore $\inv(v^jx_j)=\inv(x_jv^j)+\ell(v^j)$: we deduce that
	\begin{align*}
		\inv(\gamma^G_{\sigma_{i+1}}(w^i)\sigma_{i+1})&=\inv(\gamma^G_{\sigma_{i+1}}(w^i))+|R^{w^i}_{\sigma_{i+1}}|+|\{j\in [i]|w_j^i>\sigma_{i+1}\text{ and } \{\sigma_{i+1},w^i_j\}\in E \}|\\
		&=\inv(w^i)+ \sum_{i=1}^{k}\ell(v^i)+|R^{w^i}_{\sigma_{i+1}}| +|\{j\in [i]|w_j^i>\sigma_{i+1}\text{ and } \{\sigma_{i+1},w^i_j\}\in E \}|\\
		&=\inv(w^i)+ |L^{w^i}_{\sigma_{i+1}}|+|R^{w^i}_{\sigma_{i+1}}| +|\{j\in [i]|w_j^i>\sigma_{i+1}\text{ and } \{\sigma_{i+1},w^i_j\}\in E \}|\\
		&=\inv(w^i)+i+|\{j\in [i]|w_j^i>\sigma_{i+1}\text{ and } \{\sigma_{i+1},w^i_j\}\in E \}|.
	\end{align*}

	For the case $b)$, we start again by computing $\gamma^G_{\sigma_{i+1}}(w^i)$, hence decomposing $w^i=v^1x_1\ldots v^kx_k$ with $x_k:=\sigma_{i}$, where for every $j\in [k]$ we have $x_j\ngtr  \sigma_{i+1}$ or $\{x_j,\sigma_{i+1}\}\in E$, and $v^j$ is a word in elements $z$ that are greater than $\sigma_{i+1}$ and such that $\{z,\sigma_{i+1}\}\notin E$. Now Remark~\ref{remark:order_P} implies that for every $j\in [k]$ we have $\inv(v^jx_j)=\inv(x_jv^j)-\ell(v^j)$ because the letters of $v^j$ are all greater than $x_j$. We deduce that 
	\begin{align*}
		\inv(\gamma^G_{\sigma_{i+1}}(w^i)\sigma_{i+1})&=\inv(\gamma^G_{\sigma_{i+1}}(w^i))+|R^{w^i}_{\sigma_{i+1}}|+|\{j\in [i]|w_j^i>\sigma_{i+1}\text{ and } \{\sigma_{i+1},w^i_j\}\in E \}|\\
		&=\inv(w^i)-\sum_{i=1}^{k}\ell(v^i)+|R^{w^i}_{\sigma_{i+1}}|+|\{j\in [i]|w_j^i>\sigma_{i+1}\text{ and } \{\sigma_{i+1},w^i_j\}\in E \}|\\
		&=\inv(w^i)-|R^{w^i}_{\sigma_{i+1}}|+|R^{w^i}_{\sigma_{i+1}}|+|\{j\in [i]|w_j^i>\sigma_{i+1}\text{ and } \{\sigma_{i+1},w^i_j\}\in E \}|\\
		&=\inv(w^i)+|\{j\in [i]|w_j^i>\sigma_{i+1}\text{ and } \{\sigma_{i+1},w^i_j\}\in E \}|.
	\end{align*}

	We use the results from case $a)$ and case $b)$ to write:
	\begin{align*}
		\inv(\varphi_{G}(\sigma))=&\inv\left(\gamma^G_{\sigma_n}\left( \gamma^G_{\sigma_{n-1}}\left(\ldots\left(\gamma_{\sigma_2}^G(\gamma_{\sigma_1}^G(\epsilon)\sigma_1)\sigma_2\right)\ldots\right) \sigma_{n-1}\right)\sigma_n\right)\\
		=&\left(\sum_{\sigma_i>\sigma_{i+1}\text{ and }\{\sigma_i,\sigma_{i+1}\}\notin E}i\right)+\sum_{i=1}^{n-1}|\{j\in [i]|w_j^i>\sigma_{i+1}\text{ and } \{\sigma_{i+1},w^i_j\}\in E \}|\\
		=&\widetilde{\maj}_{G^c}(\sigma)+\widetilde{\inv}_{G}(\sigma).
	\end{align*}
\end{proof}

\section{The main identity}

In this section we show how an identity of Carlsson and Mellit \cite{Carlsson-Mellit-ShuffleConj-2015} proved for Dyck graphs extends to the case of interval graphs. This is the main result of the present article.

Recall from Section~\ref{sec:qsym} the involutions $\psi$ and $\rho$, the plethysm of quasisymmetric functions, and all the other notation for compositions.

\begin{theorem} \label{thm:main_theorem}
Let $G=([n],E)\in\mathcal{IG}_n$ be an interval graph. Then
\begin{equation} \label{eq:main_identity}
(1-q)^n \rho \left(\psi \chi_G\left[X\frac{1}{1-q}\right]\right)=\mathrm{LLT}_G[X;q].
\end{equation}
\end{theorem}
\begin{remark}
This is really an extension of \cite{Carlsson-Mellit-ShuffleConj-2015}*{Proposition~3.5}. Indeed, when $G$ is a Dyck graph, $\chi_G[X;q]$ is symmetric (by \cite{Shareshian_Wachs_Advances}*{Theorem~4.5}), the plethysm reduces to the usual plethysm of symmetric functions (cf.\ \cite{Loehr-Remmel-plethystic-2011}), $\rho$ fixes the symmetric functions while $\psi$ gives the usual $\omega$ involution of symmetric functions (cf.\ Remark~\ref{rmk:psi_restrict_to_omega}), and $\mathrm{LLT}_G[X;q]$ is precisely the unicellular LLT symmetric function corresponding to the Dyck graph $G$ (denoted $\chi(\pi)$ in \cite{Carlsson-Mellit-ShuffleConj-2015}, where $\pi$ refers to the Dyck path corresponding to the Dyck graph $G$). So, when $G$ is a Dyck graph, our \eqref{eq:main_identity} is just a rewriting of \cite{Carlsson-Mellit-ShuffleConj-2015}*{Proposition~3.5}.
\end{remark}
\begin{proof}
Using \eqref{eq:cor_chiG_coinv}, we have
\begin{align*}
\psi \chi_G\left[X\frac{1}{1-q}\right] & = \sum_{\sigma\in\mathfrak{S}_n}q^{\mathsf{coinv}_G(\sigma^{-1})}\psi L_{n,\mathsf{Des}_G(\sigma)}\left[X\frac{1}{1-q}\right]\\
& = \sum_{\sigma\in\mathfrak{S}_n}q^{\mathsf{coinv}_G(\sigma^{-1})}  L_{n,[n-1]\setminus \mathsf{Des}_G(\sigma)}\left[X\frac{1}{1-q}\right]\\
\text{(using \eqref{eq:Destilde})}& = \sum_{\sigma\in\mathfrak{S}_n}q^{\mathsf{coinv}_G(\sigma^{-1})}  L_{n,n- \widetilde{\mathsf{Des}}_G(\overline{\sigma})}\left[X\frac{1}{1-q}\right]\\
\text{(using \eqref{eq:alphaG_definition})}& = \sum_{\sigma\in\mathfrak{S}_n}q^{\mathsf{coinv}_G(\sigma^{-1})}  L_{\alpha_G(\overline{\sigma})^r}\left[X\frac{1}{1-q}\right].
\end{align*}
Setting (cf.\ \eqref{eq:invtilde})
\[c_{\alpha,P}(q):=\mathop{\sum_{\sigma\in\mathfrak{S}_n}}_{\alpha_G(\overline{\sigma})=\alpha}q^{\mathsf{coinv}_G(\sigma^{-1})}=\mathop{\sum_{\sigma\in\mathfrak{S}_n}}_{\alpha_G(\overline{\sigma})=\alpha}q^{\widetilde{\mathsf{inv}}_G(\overline{\sigma})},\]
and recalling the notations from Section~\ref{sec:key_formula}, we compute
\begin{align*}
& (1-q)^n \rho \left(\psi \chi_G\left[X\frac{1}{1-q}\right]\right)=\\
& =(1-q)^n\sum_{\sigma\in\mathfrak{S}_n}q^{\mathsf{coinv}_G(\sigma^{-1})}\rho\left(L_{\alpha_G(\sigma)^r}\left[X\frac{1}{1-q}\right]\right)\\
& =(1-q)^n\sum_{\alpha\vDash n} c_{\alpha,P}(q) \rho\left(L_{\alpha^r}\left[X\frac{1}{1-q}\right]\right)\\
\text{(using \eqref{eq:L_Cauchy})}& =(1-q)^n\sum_{\alpha\vDash n} c_{\alpha,P}(q)\rho\left(\sum_{\beta\vDash n}\prod_{i=1}^{\ell(\beta)}L_{\gamma^i(\alpha^r,\beta^r)}\left[\frac{1}{1-q}\right] M_{\beta^r}[X]\right)\\
& =(1-q)^n\sum_{\alpha\vDash n} c_{\alpha,P}(q)\sum_{\beta\vDash n}\prod_{i=1}^{\ell(\beta)}L_{\gamma^i(\alpha^r,\beta^r)}\left[\frac{1}{1-q}\right] M_\beta[X]\\
& =(1-q)^n\sum_{\alpha\vDash n} c_{\alpha,P}(q)\sum_{\beta\vDash n}\prod_{i=1}^{\ell(\beta)}L_{\gamma^i(\alpha,\beta)^r}\left[\frac{1}{1-q}\right] M_\beta[X]\\
\text{(using \eqref{eq:L_plethysm})}& =(1-q)^n\sum_{\alpha\vDash n} c_{\alpha,P}(q)\sum_{\beta\vDash n}\prod_{i=1}^{\ell(\beta)}q^{\eta((\gamma^i(\alpha,\beta)^r)^r)}h_{\beta_i}\left[\frac{1}{1-q}\right] M_\beta[X]\\
& =(1-q)^n\sum_{\alpha\vDash n} c_{\alpha,P}(q)\sum_{\beta\vDash n}q^{\eta(\gamma(\alpha,\beta))} h_{\beta}\left[\frac{1}{1-q}\right] M_\beta[X]\\
\text{(using \eqref{eq:hn_ps})}& =\sum_{\alpha\vDash n} c_{\alpha,P}(q)\sum_{\beta\vDash n}q^{\eta(\gamma(\alpha,\beta))} \frac{1}{[\beta_1]_q!\cdots [\beta_{\ell(\beta)}]_q!} M_\beta[X]\\
& =\sum_{\beta\vDash n}\sum_{\alpha\vDash n} c_{\alpha,P}(q)q^{\eta(\gamma(\alpha,\beta))} \frac{1}{[\beta_1]_q!\cdots [\beta_{\ell(\beta)}]_q!} M_\beta[X]\\
& =\sum_{\beta\vDash n} \sum_{w\in R(\beta)}q^{\inv_G(w)} M_\beta[X] ,
\end{align*}
where in the last equality we used Theorem~\ref{thm:MainFormula}.
 
Now recall that 
\[  \mathrm{LLT}_G[X;q]=\sum_{\kappa \in \mathsf{C}(G)}q^{\mathsf{inv}_G(\kappa)}x_\kappa\]
is a quasisymmetric function (see Theorem~\ref{thm:LLT_fundamental}), hence the coefficient of $M_\beta[X]$ in its expansion in the monomial basis is the coefficient of the monomial $x^\beta:=x_1^{\beta_1}x_2^{\beta_2}\cdots $, which is apparently 
\[\sum_{w\in R(\beta)}q^{\inv_G(w)}, \]
completing the proof of \eqref{eq:main_identity}.
\end{proof}

\section{Expansions in the $\Psi_\alpha$}

In \cite{Ballantine_et_al} the authors study a family of quasisymmetric functions that they call \emph{type 1 quasisymmetric power sums}, and they denote $\Psi_\alpha$. Actually $\{\Psi_\alpha\mid \alpha\text{ composition}\}$ is a basis of $\mathrm{QSym}$, and these quasisymmetric functions refine the power symmetric functions, i.e.\ for any partition $\lambda\vdash n$
\begin{equation} \label{eq:psi_plambda}
\mathop{\sum_{\alpha\vDash n}}_{\lambda(\alpha)=\lambda}\Psi_\alpha=p_\lambda,
\end{equation}
where $\lambda(\alpha)$ the unique partition obtained by rearranging in weakly decreasing order the parts of $\alpha$, and the $p_\lambda=p_{\lambda_1}p_{\lambda_2}\cdots$ are the usual \emph{power symmetric functions}.

For example $p_{(2,2,1)}=\Psi_{(2,2,1)}+\Psi_{(2,1,2)}+\Psi_{(1,2,2)}$.

The following definitions appear in \cite{Shareshian_Wachs_Advances}*{Section~7}, in a slightly different language. 

Given $G=([n],E)$ a graph and $\sigma\in\mathfrak{S}_n$ a permutation, we say that $r\in [n]$ is a \emph{left-to-right $G$-maximum} if for every $s\in [r-1]$ we have $\sigma(s)<\sigma(r)$ and $\{\sigma(s),\sigma(r)\}\notin E$. Notice that $1$ is always a left-to-right $G$-maximum, that we call \emph{trivial}. We say that $i\in [n-1]$ is a \emph{$G$-descent} if $i\in \widetilde{\mathsf{Des}}_G(\sigma)$, i.e.\ $\sigma(i)>\sigma(i+1)$ and $\{\sigma(i),\sigma(i+1)\}\notin E$. 

Given a composition $\alpha=(\alpha_1,\alpha_2,\dots,\alpha_k)\vDash n$, let $\mathcal{N}_{G,\alpha}$ be the set of $\sigma\in \mathfrak{S}_n$ such that if we break $\sigma=\sigma(1)\sigma(2)\cdots \sigma(n)$ into contiguous segments of lengths $\alpha_1,\alpha_2,\dots,\alpha_k$, each contiguous segment has neither a $G$-descent nor a nontrivial left-to-right $G$-maximum. 

For example if $G=([7],E:=\{(1,2),(1,3),(1,4),(2,3),(3,4),(4,5),(4,6),(5,6),(5,7)\})$ and $\alpha=(1,2,1,3)\vDash 7$, then $\sigma=2567143\in \mathcal{N}_{G,\alpha}$: the contiguous segments are $2$, $56$, $7$ and $143$, now the segment $56$ does not have a nontrivial left-to-right $G$-maximum, since $\{5,6\}\in E$, the segment $143$ does not have a nontrivial left-to-right $G$-maximum, since $\{1,4\}\in E$, nor a $G$-descent, since $\{3,4\}\in E$. On the other hand $\tau=5267143\notin \mathcal{N}_{G,\alpha}$: the contiguous segments are $5$, $26$, $7$ and $143$, but now the segment $26$ has a nontrivial left-to-right $G$-maximum, since $\{2,6\}\notin E$.

Given a composition $\alpha$, define $z_\alpha:=z_{\lambda(\alpha)}$, where, as usual, for every partition $\lambda\vdash n$, if $m_i$ denotes the number of parts of $\lambda$ equal to $i$, then $z_\lambda:=\prod_{i=1}^nm_i!\cdot i^{m_i}$.

Finally, recall the involution $\omega:\mathrm{QSym}\to \mathrm{QSym}$ from Section~\ref{sec:qsym}.

We state our first conjecture.

\begin{conjecture} \label{conj:interval_psi_exp}
	For any interval graph $G=([n],E)$ we have
	\[\omega \chi_G[X;q]=\sum_{\alpha\vDash n}\frac{\Psi_\alpha}{z_{\alpha}}\sum_{\sigma \in \mathcal{N}_{G,\alpha}}q^{\widetilde{\mathsf{inv}}_G(\sigma)}.\]
\end{conjecture}

This conjecture was inspired by the following formula, that was conjecture by Shareshian and Wachs \cite{Shareshian_Wachs_Advances}*{Conjecture~7.6} and later proved by Athanasiadis \cite{Athanasiadis}.
\begin{theorem} \label{thm:athanasiadis}
	For any Dyck graph $G=([n],E)$ we have
\[\omega \chi_G[X;q]=\sum_{\lambda\vdash n}\frac{p_\lambda}{z_{\lambda}}\sum_{\sigma \in \mathcal{N}_{G,\lambda}}q^{\widetilde{\mathsf{inv}}_G(\sigma)}.\]
\end{theorem}
Notice that thanks to Theorem~\ref{thm:athanasiadis} and \eqref{eq:psi_plambda}, our Conjecture~\ref{conj:interval_psi_exp} is equivalent to the following tempting conjecture.
\begin{conjecture}
For any Dyck graph $G=([n],E)$ and any composition $\alpha\vDash n$ we have
\[ \sum_{\sigma \in \mathcal{N}_{G,\alpha}}q^{\widetilde{\mathsf{inv}}_G(\sigma)}=\sum_{\sigma \in \mathcal{N}_{G,\lambda(\alpha)}}q^{\widetilde{\mathsf{inv}}_G(\sigma)}. \]
\end{conjecture}

In fact, we have a more general conjecture involving our quasiymmetric functions $\mathcal{Q}_{F}^{(G)}$.

\begin{conjecture}  \label{conj:forest_psi_expansion}
	For any interval graph $G=([n],E)$ and any permutation $\tau\in \mathfrak{S}_n$ we have
	\begin{equation} \label{eq:psi_forest}
	\omega \mathcal{Q}_{\Phi_G(\tau)}^{(G)}=\sum_{\alpha\vDash n}\frac{\Psi_\alpha}{z_{\alpha}}\# \{ \sigma \in \mathcal{N}_{G,\alpha}\mid \mathsf{CoInv}_G(\sigma^{-1})=\mathsf{Inv}_G(\tau)\}.
\end{equation}
\end{conjecture}

\begin{lemma}
Conjecture~\ref{conj:forest_psi_expansion} implies Conjecture~\ref{conj:interval_psi_exp}.
\end{lemma}
\begin{proof}
Given an interval graph $G=([n],E)\in\mathcal{IG}_n$, Proposition~\ref{prop:Phi_Coinv_inverval} implies that $\mathsf{wt}_G(\Phi_G(\tau))=\mathsf{coinv}_G(\tau)$, while \eqref{eq:invtilde} implies $\widetilde{\inv}_G(\sigma)=\inv_G(\sigma^{-1})$. Now observe that $\sigma\in \mathfrak{S}_n\subset \mathsf{PC}(G)$, hence $\mathsf{CoInv}_G(\sigma)=E\setminus \mathsf{Inv}_G(\sigma)$, so if $\mathsf{CoInv}_G(\sigma^{-1})=\mathsf{Inv}_G(\tau)$, then
\[\inv_G(\sigma^{-1})=|E|- \mathsf{coinv}_G(\sigma^{-1})=|E|- \inv_G(\tau)=\mathsf{coinv}_G(\tau).\]

Hence multiplying Conjecture~\ref{conj:forest_psi_expansion} by $q^{\mathsf{wt}_G(\Phi_G(\tau))}$ we get
\begin{align*}
q^{\mathsf{wt}_G(\Phi_G(\tau))}\omega \mathcal{Q}_{\Phi_G(\tau)}^{(G)} & =q^{\mathsf{wt}_G(\Phi_G(\tau))}\sum_{\alpha\vDash n}\frac{\Psi_\alpha}{z_{\alpha}}\# \{ \sigma \in \mathcal{N}_{G,\alpha}\mid \mathsf{CoInv}_G(\sigma^{-1})=\mathsf{Inv}_G(\tau)\}\\
& =\sum_{\alpha\vDash n}\frac{\Psi_\alpha}{z_{\alpha}}\mathop{\sum_{\sigma \in \mathcal{N}_{G,\alpha}}}_{ \mathsf{CoInv}_G(\sigma^{-1})=\mathsf{Inv}_G(\tau)}q^{\mathsf{coinv}_G(\tau)}\\
& =\sum_{\alpha\vDash n}\frac{\Psi_\alpha}{z_{\alpha}}\mathop{\sum_{\sigma \in \mathcal{N}_{G,\alpha}}}_{ \mathsf{CoInv}_G(\sigma^{-1})=\mathsf{Inv}_G(\tau)}q^{\widetilde{\mathsf{inv}}_G(\sigma)}.
\end{align*}

Now notice that because of Theorem~\ref{thm:PhiG_surjective}, the set
\[\mathsf{Inv}(G):=\{\mathsf{Inv}_G(\sigma)\mid \sigma\in \mathfrak{S}_n\}\]
is such that
\[\mathsf{Inv}(G)=\{E\setminus \mathsf{CoInv}_G(\sigma)\mid \sigma\in f_G(\mathsf{ISF}(G))\}=\{\mathsf{Inv}_G(\sigma)\mid \sigma\in f_G(\mathsf{ISF}(G))\}.\]

Hence, using Theorems~\ref{thm:chrom_forests_formula} and \ref{thm:PhiG_surjective}, we have
\begin{align*}
\omega \chi_G[X;q]& = \sum_{F\in \mathsf{ISF}(G)}q^{\mathsf{wt}_G(F)} \omega \mathcal{Q}_F^{(G)}\\
& = \sum_{\tau\in f_G(\mathsf{ISF}(G))}q^{\mathsf{wt}_G(\Phi_G(\tau))} \omega  \mathcal{Q}_{\Phi_G(\tau)}^{(G)}\\
& = \sum_{\tau\in f_G(\mathsf{ISF}(G))}\sum_{\alpha\vDash n}\frac{\Psi_\alpha}{z_{\alpha}}\mathop{\sum_{\sigma \in \mathcal{N}_{G,\alpha}}}_{ \mathsf{CoInv}_G(\sigma^{-1})=\mathsf{Inv}_G(\tau)}q^{\widetilde{\mathsf{inv}}_G(\sigma)}\\
& = \sum_{\alpha\vDash n}\frac{\Psi_\alpha}{z_{\alpha}} \sum_{\sigma \in \mathcal{N}_{G,\alpha}} q^{\widetilde{\mathsf{inv}}_G(\sigma)}
\end{align*}
which is precisely Conjecture~\ref{conj:interval_psi_exp}.

\end{proof}

\section{Speculative comments}

In the light of the results and conjectures of the present article, it is natural to wonder how far both conjectures and results about Dyck graphs can be extended to interval graphs. We add here some speculative comments.

First of all notice that the results in this article are ``tight'', in the sense that, experimentally, as soon as $G$ is not an interval graph, our statements of both theorems and conjectures tend to be false for $G$.

About the $e$-positivity conjecture of Shareshian and Wachs for Dyck graphs: we wonder if there exists a quasisymmetric refinement of the elementary basis for which the positivity conjecture extends to interval graphs. This would extend the Shareshian-Wachs conjecture to a quasisymmetric situation, and it could possibly suggest a new approach to the original conjecture. Similarly, we wonder if such a refinement would give a positivity of the quasisymmetric functions $\mathrm{LLT}_G[X;q+1]$ for $G$ an interval graph, in analogy with the results in \cite{Alexandersson_Sulzgruber}. 

We wonder if there is an extension of the formula of Abreu and Nigro \cite{Abreu_Nigro_Forests} for $\chi_G[X;q]$ to interval graphs $G$. More generally, we wonder if there is a nice way to relate their formula to our Theorem~\ref{thm:chrom_forests_formula}, as they have some obvious similarities.

Finally, the identity \eqref{eq:CM_identity} of Carlsson and Mellit is somehow related to the equivariant cohomology of Hessenberg varieties. These varieties are associated to so called \emph{Hessenberg vectors} (or \emph{Hessenberg functions}), which correspond naturally to Dyck paths, hence to Dyck graphs. Our interval graphs are similarly associated to vectors that are analogous to Hessenberg vectors. Unfortunately the natural way to associate a variety to these vectors (i.e.\ following the definition of the Hessenberg varieties) does not seem to give new interesting geometric objects. Nonetheless, one can construct a ``cohomology ring'' following GKM theory (see e.g.\ \cite{guaypaquet_second_pf}*{Section~8}). It would be interesting to see if these rings share some properties with the ones associated to Dyck graphs. In particular if they are related to our Theorem~\ref{thm:main_theorem}.

\bibliographystyle{amsalpha}
\bibliography{Biblebib}

\end{document}